\documentclass[reqno, 11pt, a4paper]{amsart}

\usepackage{mathrsfs}
\usepackage{physics}
\usepackage{esint}
\usepackage[super]{nth}
\usepackage{latexsym,ifthen,xspace}
\usepackage{amsmath,amssymb,amsthm}
\usepackage{mathtools}
\usepackage{enumerate, calc, bbm}
\usepackage[dvipsnames]{xcolor}
\usepackage[colorlinks=true, pdfstartview=FitV, linkcolor=blue,
  citecolor=blue, urlcolor=blue, pagebackref=false]{hyperref}
\usepackage[text={33pc,605pt},centering]{geometry}    
\usepackage{orcidlink}

\usepackage{graphicx}
\usepackage{marginnote}
\usepackage{soul}

\DeclareMathOperator*\uplim{\overline{lim}}

\newtheorem{theorem}{Theorem}[section]
\newtheorem{lemma}[theorem]{Lemma}
\newtheorem{prop}[theorem]{Proposition}
\newtheorem{assumption}[theorem]{Assumption}

\theoremstyle{definition}
\newtheorem{definition}[theorem]{Definition}

\theoremstyle{remark}
\newtheorem{remark}[theorem]{Remark}

\numberwithin{equation}{section}

\DeclareMathAlphabet{\mathsl}{OT1}{cmss}{m}{sl}
\SetMathAlphabet{\mathsl}{bold}{OT1}{cmss}{bx}{sl}

\newcommand\restrict[1]{\raisebox{-.5ex}{$|$}_{#1}}
\newcommand{\Norm}[2]{%
  \ensuremath{%
    \mathchoice{\big\lVert #1 \big\rVert}
     {\lVert #1 \rVert}
     {\lVert #1 \rVert}
     {\lVert #1 \rVert}_{\raisebox{-.0ex}{$\scriptstyle #2$}}
  }
}

\DeclareMathOperator{\mean}{\mathbb{E}}
\DeclareMathOperator{\Mean}{\mathrm{E}}
\DeclareMathOperator{\prob}{\mathbb{P}} 
\DeclareMathOperator{\Prob}{\mathrm{P}} 

\DeclareMathOperator{\supp}{\mathrm{supp}}

\newcommand{\ldef}{\ensuremath{\mathrel{\mathop:}=}}
\newcommand{\rdef}{\ensuremath{=\mathrel{\mathop:}}}

\usepackage{mathtools}

\DeclarePairedDelimiter\floor{\lfloor}{\rfloor}
\DeclarePairedDelimiterX{\inp}[2]{\langle}{\rangle}{#1, #2} 
\DeclarePairedDelimiterX{\inpe}[1]{\langle\,}{\,\rangle}{#1} 

\begin{document}

\title[Invariance principle for cyclic random walks]{Quenched invariance principle for random walks in random environments admitting a cycle decomposition}


\author[J.-D.\ Deuschel]{Jean-Dominique Deuschel}
\address{Technische Universit\"at Berlin}
\curraddr{Strasse des 17. Juni 136, 10623 Berlin, Germany}
\email{deuschel@math.tu-berlin.de}
\thanks{}

\author[M.\ Slowik]{Martin Slowik \orcidlink{0000-0001-5373-5754}}
\address{University of Mannheim}
\curraddr{Mathematical Institute, B6, 26, 68159 Mannheim}
\email{slowik@math.uni-mannheim.de}
\thanks{}

\author[W.\ Weng]{Weile Weng \orcidlink{0009-0009-6565-878X}}
\address{Technische Universit\"at Berlin}
\curraddr{Strasse des 17. Juni 136, 10623 Berlin, Germany}
\email{weile.weng@gmail.com}
\thanks{Supported by DFG-funded Berlin-Oxford IRTG 2544, and Berlin Mathematical School.}

\subjclass[2000]{60K37, 60F17, 82C41, 82B43}

\keywords{Random walks in random environments, cycle representation, invariance principle
    }

\date{\today}

\dedicatory{}

\begin{abstract}
  We study a class of non-reversible, continuous-time random walks in random environments on $\mathbb{Z}^d$ that admit a cycle representation with finite cycle length.  The law of the transition rates, taking values in $[0, \infty)$, is assumed to be stationary and ergodic with respect to space shifts. Moreover, the transition rate from $x$ to $y$, denoted by $c^\omega(x,y)$, is a superposition of non-negative random weights on oriented cycles that contain the edge $(x,y)$. We prove a quenched invariance principle under moment conditions that are comparable to the well-known p-q moment condition of Andres, Deuschel, and Slowik \cite{ADS15} for the random conductance model. A key ingredient in proving the sublinearity is an energy estimate for the non-symmetric generator. Our result extends that of Deuschel and Kösters \cite{DK08} beyond strong ellipticity and bounded cycle lengths.
\end{abstract}

\maketitle

\tableofcontents

\section{Introduction}
\label{sec:intro}


We consider time continuous random walks on $\mathbb{Z}^d$ in  random environments with a finite but unbounded cycle decomposition.
Random walks with finite cycle decomposition, also called centred walks
cf. \cite{Mat06},  are the discrete versions of  diffusions on $\mathbb{R}^d$ in divergence form generated by
$$Lf(x)=\text{div}(c(x)\nabla f(x)).$$
where $c(x)=s(x)+a(x) $, $s(x)=s^*(x)$ is the symmetric, $a^*(x)=-a(x)$
the anti-symmetric part of the matrix $c(x)\in \mathbb{R}^{d\times d}.$
Assuming uniform ellipticity and boundedness:
$$\lambda<s(x)<\Lambda, \qquad |a(x)|<B, \qquad \lambda, \Lambda, B\in (0, \infty)$$
the celebrated results of Moser-De Giorgi-Nash-Aronson show 
parabolic and elliptic  Harnack inequalities and Gaussian type estimates
for the corresponding heat kernel, cf. \cite{Aro67, Nas58, Mos64, DeG57}.
Moreover, Papanicolaou and Varadhan \cite{PV82} proved  a quenched invariant principle (QIP),  i.e. the almost sure convergence for the diffusively rescaled
diffusion in ergodic environment.  See also Fannjiang and Komorowski
\cite{FK97, FK99} for the unbounded case.

Random walks with finite cycle decomposition generated by
$$Lf(x)=\sum_{y} c(x,y)(f(y)-f(x)), $$
have jump rates $c(x,y)\ge 0$  given in terms of oriented cycles
$\gamma=(x_0,x_1,..., x_n=x_0)$ of length $\ell(\gamma)=n$ and weights
$w(\gamma)\ge 0$,
$$c(x,y)=\sum_{\gamma} w(\gamma)1_{(x,y)\in\gamma}.$$
The corresponding (non-symmetric) Dirichlet form is determined by
$${\mathcal{E}}(f,g)=\sum_\gamma
w(\gamma){\mathcal{E}}_\gamma(f,g), \qquad\text{where}\quad 
{\mathcal{E}}_\gamma(f,g)=\sum_{x_i\in\gamma}f(x_i)(g(x_i)-g(x_{i+1})).$$
In particular the counting measure is invariant, and the adjoint generator
${\mathcal{E}}^*(f,g)={\mathcal{E}}(g,f)$
is given in terms of the reversed cycles
$\gamma^*=(x_n,x_{n-1},...,x_0=x_n)$ with weight
$w^*(\gamma^*)=w(\gamma)$.
In the special case that all the  cycles have  length $\ell(\gamma)=2$,
then $\gamma=\gamma^*$, thus  the random walk is reversible  with
symmetric jump rates, also called conductances.

Deuschel and Kumagai \cite{DK13} showed that diffusions in divergence
form can be approximate by random walks  on $\frac{1}{n}\mathbb{Z}^d$ with bounded cycle decomposition and  bounded weights.
In an earlier work by Deuchel and K\"{o}sters \cite{DK08}, the quenched invariant principle (QIP) is derived for 
uniformly elliptic  random walks in ergodic random environments with
bounded cycle decomposition and bounded weights.
In the symmetric situation, i.e. dealing with cycles of length 2, the QIP
has been proved for non-elliptic random conductance models under suitable
moment conditions, cf. \cite{ADS15, BS20}.

The  objective of the current paper is to obtain the QIP  in the setting
of unbounded cycle length,  and non-uniform elliptic, i.e. unbounded and
degenerate weights under appropriate moment conditions.
In particular our result extends the QIP for the  reversible  random
conductance model in degenerated ergodic environment to the 
non-symmetric  case under similar moment conditions.
When the cycle lengths are unbounded, the random walk
does not satisfy the standard sector condition
$$({\mathcal{E}}(f,g))^2\leq\text{const}\, {\mathcal{E}}(f,f) \cdot {\mathcal{E}}(g,g), $$ which is usually assumed for the derivation of the invariance principle
in non-symmetric situations, cf. the works of Varadhan and collaborators \cite{Var95, SVY00}.

Our proof of QIP follows the general strategy of Andres, Deuschel and Slowik \cite{ADS15}, which deals with the symmetric case in the setting of a non-uniform elliptic conductance model. Specifically, we first introduce harmonic coordinates and then construct the corrector. Next, using PDE techniques adapted to the discrete setting we show the almost sure sublinearity of the corrector. The main challenge here is to deal and control the non-symmetric part of the generator. In particular in view the unboundedness of the cycle length
and loss of the sector condition, the construction of the corrector by means of the Lax-Milgram theorem necessitates  a moment condition on the  weighted length of cycles, which implies a bound of of the  ${\mathcal{H}}_{-1}$ norm of the corresponding
drift -- a natural condition in the setting of double stochastic generators, cf. \cite{KT17}.  The next step is to control the corrector. Instead of the Moser iteration used in \cite{ADS15}, we adapt the
analytical technique of the De Giorgi iteration, which is robust for parabolic extension.  While the Sobolev and weighted Poincare inequalities relying  on the symmetric part are identical, the corresponding  energy estimate is in view of the
non-symmetry more challenging. Here again the moment condition on the
weighted length is implemented.

\section{Model and Results}
\label{sec:model-results}
\subsection{The model}
\label{subsec:model}
In this section, we begin with introducing some terminology related to the structure of the underlying graph and the cyclic random environments that we are interested in. We use more elaborate notations for clarity. 

Consider the $d$-dimensional Euclidean lattice, $\mathbb{Z}^{d}$, for $d \geq 1$ that we view as a simple, non-oriented lattice graph $(\mathbb{Z}^{d}, E_{d})$ with vertex set $\mathbb{Z}^{d}$ and non-oriented, nearest-neighbor edge set $E_{d} \ldef \bigl\{ \{x, y\} : x, y \in \mathbb{Z}^{d} \text{ s.th. } |x-y| = 1 \bigr\}$, where $|\cdot|$ stands for the usual Euclidean norm on $\mathbb{Z}^{d}$.  An oriented (directed) edge $(x, y)$ is defined by the choice of an ordering in the edge $\{x, y\} \in E_{d}$, and the corresponding set of all oriented edges is denoted by $\vec{E}_{d}$.  Note that we always view $\vec{E}_{d}$ as a symmetric\footnote{Here, symmetric means that $(x, y) \in \vec{E}_{d}$ iff $(y, x) \in \vec{E}_{d}$.} subset of $\mathbb{Z}^{d} \times \mathbb{Z}^{d}$.  Further, we define an involution on $\vec{E}_{d}$ by the mapping $ (x, y) \mapsto -(x, y) \ldef (y, x)$.

An oriented (directed) path $\vec{\gamma}$ on $\mathbb{Z}^{d}$ from a vertex $v_{0}$ to a vertex $v_{n}$ is defined, in graph-theoretic sense, as an alternating sequence $\vec{\gamma} = (v_{0}, \vec{e}_{1}, \ldots, v_{n-1}, \vec{e}_{n}, v_{n})$ of mutually distinct vertices $v_{j} \in \mathbb{Z}^{d}$ (except possibly the initial and final ones) and edges $\vec{e}_{j} = (v_{j-1}, v_{j}) \in \vec{E}_{d}$.  Such path has length $|\vec{\gamma}| = n$.  By writing $x \in \vec{\gamma}$, we mean that the vertex $x$ belongs to the oriented path $\vec{\gamma}$, that is, there exists $j \in \{0, \ldots, n\}$ such that $x = v_{j}$. Likewise, we write $(x, y) \in \vec{\gamma}$ if the oriented edge $(x, y)$ belongs to the oriented path $\vec{\gamma}$.  Further, for any $z \in \mathbb{Z}^{d}$, the shifted path $z + \vec{\gamma}$ from $z + v_{0}$ to $z + v_{n}$ is defined by $(z + v_{0}, z + \vec{e}_{1}, \ldots, z + v_{n-1}, z + \vec{e}_{n}, z + v_{n})$, where $z + (x, y) \ldef (z + x, z + y)$ for any $(x, y) \in \vec{E}_{d}$.
Finally, an oriented (directed) cycle on $\mathbb{Z}^{d}$ based at $x_{0}$ is, by definition, a finite, oriented path $\vec{\gamma}$ that begins and end at the same vertex $x_{0}$.

In what follows, let $\Gamma_{\!0} \ldef \{\vec{\gamma}_{i} : i \in \mathbb{N}\}$ be a collection of directed cycles on $\mathbb{Z}^{d}$ based at $0$ of finite length such that
\begin{align}
  \label{eqn:cycle_cover_nns} 
  \bigl\{
    x : \exists\, \vec{\gamma} \in \Gamma_{\!0}
    \text{ s.th. either }
    (0, x) \in \vec{\gamma} \text{ or } (x, 0) \in \vec{\gamma}
  \bigr\}
  \;=\;
  \bigl\{ x : |x| = 1 \bigr\}
  \;\rdef\;
  \mathcal{N}, 
\end{align}
and $(\Omega, \mathcal{F}) \ldef \bigl([0, \infty)^{\Gamma_{\!0} \times \mathbb{Z}^{d}}, \mathcal{B}([0, \infty)^{\otimes (\Gamma_{\!0} \times \mathbb{Z}^{d})})\bigr)$ be a measurable space equipped with the Borel-$\sigma$-algebra. For any $\omega \in \Omega$, $x \in \mathbb{Z}^{d}$ and $\vec{\gamma} \in \Gamma_{\!0}$, $\omega_{\vec{\gamma}}(x)$ is called the \emph{cycle weight} of the cycle $x + \vec{\gamma}$ that is based at $x$.  
Henceforth, we consider a probability measure, $\prob$, on $(\Omega, \mathcal{F})$, and write $\mean$ for the expectation with respect to $\prob$.  We endow the probability space $(\Omega, \mathcal{F}, \prob)$ with the $d$-parameter group of space shift $(\tau_{x})_{x \in \mathbb{Z}^{d}}$ that acts on $\Omega$ as 
\begin{align}
  \label{eq:def:shift}
  (\tau_{x} \omega_{\vec{\gamma}})(y)
  \;\ldef\;
  \omega_{\vec{\gamma}}(x + y)
  \qquad \forall\, x, y \in \mathbb{Z}^{d},\, \vec{\gamma} \in \Gamma_{\!0}.
\end{align}
\begin{assumption}\label{ass:P}
  Throughout we assume that $\prob$ satisfies following conditions:
  \begin{enumerate}[(i)]
  \item
    $\prob$ is stationary and ergodic with respect to space shifts $(\tau_{x})_{x \in \mathbb{Z}^{d}}$, that is, $\prob \circ \tau_{x}^{-1} = \prob$ for all $\mathbb{Z}^{d}$ and $\prob[A] \in \{0, 1\}$ for all $A \in \mathcal{F}$ such that $\tau_{x}^{-1}(A) = A$ for all $x \in \mathbb{Z}^{d}$.
    
  \item
    For every $z \in \mathcal{N}$,
    \begin{align}
      \label{ass:symmetric-ellipticity}
      \prob\biggl[
        \sum\nolimits_{\vec{\gamma} \in \Gamma_{0}} \sum\nolimits_{x \in \mathbb{Z}^{d}}
        \omega_{\vec{\gamma}}(x)\,
        \bigl(
          \mathbbm{1}_{(0, z) \in x + \vec{\gamma}} + \mathbbm{1}_{(z, 0) \in x + \vec{\gamma}}
        \bigr)
        > 0
      \biggr]
      \;=\;
      1.
    \end{align}
    
  \item
    $\mean\Bigl[\sum_{\vec{\gamma} \in \Gamma_{\!0}} \omega_{\vec{\gamma}}(0)\, |\vec{\gamma}| \Bigr] < \infty$.
  \end{enumerate}
\end{assumption}
Given any $\omega \in \Omega$, we define \emph{edge weights}
, $c^{\omega}$, by
\begin{align}\label{eq:def:random_edge_weights}
  c^{\omega}(x, y)
  \;\ldef\;
  \sum_{\vec{\gamma} \in \Gamma_{\!0}} \sum_{z \in \mathbb{Z}^{d}}
  \omega_{\vec{\gamma}}(z) \mathbbm{1}_{(x, y) \in z + \vec{\gamma}}
  \;\in\;
  [0, \infty],
  \qquad \forall\, x, y \in \mathbb{Z}^{d}.
\end{align}
Obviously, $c^{\omega}(x, y) = 0$ for all $(x, y) \in (\mathbb{Z}^{d} \times \mathbb{Z}^{d}) \setminus \vec{E}_{d}$ and, by Assumption~\ref{ass:P}-(iii), $c^{\omega}(x, y) < \infty$ for every $(x, y) \in \vec{E}_{d}$ and $\prob$-a.e.\ $\omega$.  As an immediate consequence of \eqref{eq:def:shift} we have that, for any $\omega \in \Omega$ and $x, y, z \in \mathbb{Z}^{d}$,
\begin{align}
  \label{prop:shifted-rate}  
  c^{\tau_{z} \omega}(x, y) \;=\; c^{\omega}(x + z, y + z).
\end{align}
It should be noted that the edge weights are symmetric, that is, $c^{\omega}(x, y) = c^{\omega}(y, x)$ iff the cycle weights are supported on cycles $\vec{\gamma} \in \Gamma_{\!0}$ with length $|\vec{\gamma}| \leq 2$.  Clearly, by construction, the edge weights, $c^{\omega}$, admits a so-called \emph{finite cycle decomposition}.  However, it is important to note that two different sets of cycles and cycle weights may give rise to the same edge weights. By setting
\begin{align*}
  c_{\mathrm{s}}^{\omega}(x, y)
  \;\ldef\;
  \frac{1}{2} \bigl( c^{\omega}(x, y) + c^{\omega}(y, x) \bigr)
  \qquad \text{and} \qquad
  c_{\mathrm{a}}^{\omega}(x, y)
  \;\ldef\;
  \frac{1}{2} \bigl( c^\omega(x, y) - c^{\omega}(y, x) \bigr)
\end{align*}
for any $x, y \in \mathbb{Z}^{d}$, we can decompose the edge weights into the symmetric and anti-symmetric part, respectively.  Notice that, in view of \eqref{eq:def:random_edge_weights}, it holds that
\begin{align*}
  \sum_{y \in \mathbb{Z}^{d}} c_{\mathrm{a}}^{\omega}(x, y) \;=\; 0
  \qquad \text{and} \qquad
  |c_{\mathrm{a}}^{\omega}(x, y)| \;\leq\; c_{\mathrm{s}}^{\omega}(x, y),
  \qquad \forall\, x, y \in \mathbb{Z}^d.
\end{align*}
Notice that the property on the left is equivalent to
\begin{align}
  \label{eqn:prop:doubly_stochastic}
  \sum_{y \in \mathbb{Z}^{d}} c^{\omega}(x, y)
  \;=\;
  \sum_{y \in \mathbb{Z}^{d}} c^{\omega}(y, x),
  \qquad \forall\, x \in \mathbb{Z}^d.
\end{align}
Further, we define following measures on $\mathbb{Z}^{d}$
\begin{align}
  \label{eq:def:random_speed_measure}
  \mu^{\omega}(x)
  \;\ldef\;
  \sum_{z \in \mathcal{N}} c_{\mathrm{s}}^{\omega}(x, x + z)
  \qquad \text{and} \qquad
  \nu^{\omega}(x)
  \;\ldef\;
  \sum_{z \in \mathcal{N}} \frac{1}{c_{\mathrm{s}}^{\omega}(x, x + z)}.
\end{align}
\begin{remark}
  Given Assumption~\ref{ass:P}-(i), and the definition \eqref{eq:def:random_edge_weights}, Assumption~\ref{ass:P}-(ii) is equivalent to
  \begin{align}
    \label{ass:symmetric-ellipticity-2}
    c_{\mathrm{s}}^{\omega}(x, y) \;>\; 0,
    \qquad \forall\, (x, y) \in \vec{E}_{d} \text{ and } \prob\text{-a.e. } \omega.
  \end{align}
  Together with definition \eqref{eq:def:random_speed_measure}, Assumption~\ref{ass:P}-(ii) and~(iii) are equivalent to
  \begin{align}
    \label{eq:cond:1st-moment}
    \prob\bigl[ \nu^{\omega}(0) < \infty \bigr] \;=\; 1
    \qquad \text{and} \qquad
    \mean\bigl[ \mu^{\omega}(0) \bigr] \;<\; \infty.
  \end{align}
\end{remark}

For any given $\omega \in \Omega$, we are interested in a continuous-time
Markov chain, $X \equiv (X_{t} : t \geq 0)$, on $\mathbb{Z}^{d}$ that is defined by the following rule: the stochastic process $X$ waits at its current position $x$ an exponential time with mean $1/\mu^{\omega}(x)$ and then jumps to a vertex $y$ that is connected to $x$ by an edge $\{x, y\} \in E_{d}$ with probability $c^{\omega}(x, y) / \mu^{\omega}(x)$.  Since the law of the holding times of $X$ depends on the spatial position and the jump rates admits a cycle representation, we refer to $X$ as both a \emph{variable speed random walk (VSRW)} and a \emph{cyclic random walk}.  Its generator, $L^{\omega}$, acts on bounded functions $f\colon \mathbb{Z}^{d} \to \mathbb{R}$ as
\begin{align}\label{eq:def:generator_quenched}
  \bigl(L^{\omega} f\bigr)(x)
  \;\ldef\;
  \sum_{y \in \mathbb{Z}^{d}} c^{\omega}(x,y) \big(f(y) - f(x))
  \qquad \forall\, x \in \mathbb{Z}^d.
\end{align}
We denote by $\Prob^{\omega}$ the law of the process on the space of $\mathbb{Z}^{d}$-valued c\`{a}dl\`{a}g functions on $\mathbb{R}$, starting in $x$. The corresponding expectations will be denoted by $\Mean^{\omega}$.  Notice that $X$ is invariant with respect to the counting measure.  Indeed, from \eqref{eqn:prop:doubly_stochastic} it follows that
\begin{align*}
  \sum_{x \in \mathbb{Z}^{d}} \bigl(L^{\omega} f \bigr)(x)
  \;=\;
  \sum_{y \in \mathbb{Z}^{d}} f(y) \sum_{x \in \mathbb{Z}^{d}} c^{\omega}(x, y)
  -
  \sum_{x \in \mathbb{Z}^{d}} f(x) \sum_{y \in \mathbb{Z}^{d}} c^{\omega}(x, y)
  \;=\;
  0
\end{align*}
for any finitely supported $f\colon \mathbb{Z}^{d} \to \mathbb{R}$. Moreover, Assumption~\ref{ass:P} is sufficient to conclude that, for $\prob$-a.e.\ $\omega$ and any $x \in \mathbb{Z}^{d}$, the stochastic process, $X$, is conservative under $\Prob_{x}^{\omega}$, that is, there are only finitely many jumps in any finite time interval $\Prob_{x}^{\omega}$-a.s., see Proposition \ref{prop:time-ergodicity-and-non-explosion} later on. Finally, we write $p^{\omega}(t, x, y) \ldef \Prob_{x}^{\omega}\bigl[ X_{t} = y \bigr]$ for $x, y \in \mathbb{Z}^{d}$ and $t \geq 0$ to denote the transition density with respect to the counting measure.
%
%
%
%


%
%
%
%

\subsection{Results}
We are interested in the $\mathbb{P}$-almost sure or quenched long-time behaviour, in particular in obtaining a quenched functional central limit theorem for the process $X$ in the sense of the following definition.

\begin{definition}[QFCLT] Set $X_t^{(n)}:=\frac{1}{n}X_{n^2t}, t\geq 0$. We say the \textit{quenched functional CLT} (QFCLT) or \textit{quenched invariance principle} holds for $X$ if for $\mathbb{P}$-a.s. $\omega$ under $\Prob_0^\omega$, $X^{(n)}$ converges in law to a Brownian motion on $\mathbb{R}^d$ with covariance matrix $\Sigma^2=\Sigma \cdot \Sigma^T$. That is, for every $T>0$ and every bounded continuous function $F$ on the Skorohod space $D([0, T], \mathbb{R}^d)$, setting $\psi_n= \Mean_0^\omega[F(X^{(n)})]$ and $\psi_\infty= \Mean_0^{\text{BM}}[F(\Sigma \cdot W)]$ with $(W, \Prob_0^{\text{BM}})$ being a Brownian motion started at $0$, we have that $\psi_n \to \psi_\infty \mathbb{P}$-a.s.

As our main result we establish a QFCLT for $X$ under some additional moment conditions. 
\end{definition}

\begin{assumption}
  \label{ass:p-q-moment}
  There exist $p, q \in (1, \infty]$ be such that $1/p + 1/q < 2/d$ and assume that
  \begin{enumerate}[(i)]
  \item
    $\mathbb{E}\bigl[\mu^{(2), \omega}(0)^{p} \bigr] < \infty$, where
    \begin{align*}
      \mu^{(k), \omega}(x)
      \;\ldef\;
      \sum_{\vec{\gamma} \in \Gamma_{0}} \sum_{y \in \mathbb{Z}^{d}} \omega_{\vec{\gamma}}(y)\abs{\vec{\gamma}}^{k}\, \mathbbm{1}_{x\in \vec{\gamma}+y},
      \qquad k \in \mathbb{N}_{0}, x \in \mathbb{Z}^{d}.
    \end{align*} 
  \item
    $\mathbb{E}\bigl[ \bigl(1/(c_{\mathrm{s}}^{\omega}(0, z)\bigr)^q \bigr]<\infty$, for all $z \in \mathcal{N}$.
  \end{enumerate}
\end{assumption}
\begin{remark}
  Assumption~\ref{ass:p-q-moment}-(i) implies
  \begin{align}
    \label{eqn:p-moment}
    \mathbb{E}\bigl[c_{\mathrm{s}}^{\omega}(0,z)^{p}\bigr]
    \;<\;
    \infty,
    \qquad \forall\, z\in \mathcal{N}.
  \end{align}
  Assumption~\ref{ass:P}-(ii) follows from \eqref{eqn:p-moment}. Moreover,  \eqref{eqn:p-moment} together with Assumption ~\ref{ass:p-q-moment}-(ii) constitutes the p-q moment condition in \cite[Theorem 1.3]{ADS15} for the case of random conductance model. In this symmetric setting, the constraint $1/p + 1/q < 2/d$ is optimal in deriving the quenched local limit theorem (QLLT) for the constant speed random walks (CSRW) \cite[Theorem 5.4]{ADS16}. In view of the QFCLT, \cite{BS20} improved the constraint to $1/p + 1/q < 2/(d-1)$, for $d \geq 3$, with $p = q = 1$ already established by Biskup \cite{Bis11} as the optimal condition for $d = 2$. Moreover, the bound $2/(d-1)$ is optimal in ensuring the corrector is sublinear everywhere, as well as for obtaining the QLLT for VSRW. We adopt the condition $1/p + 1/q < 2/d$ as its application has been prevalent in many related works cf.~\cite{ADS16, ACDS18, BCKW21, Bis23}.
\end{remark}

\begin{theorem}[QFCLT]
  \label{thm:QFCLT} 
  Suppose that $d \geq 2$ and Assumption~\ref{ass:P}-(i) and \ref{ass:p-q-moment} hold. 
  Then, the QFCLT holds for $X$ with a deterministic non-degenerate covariance matrix $\Sigma^{2}$.
\end{theorem}

\section{Quenched invariance principle}
%
%
%
%
In this section, we deal with the proof of Theorem \ref{thm:QFCLT}. We will construct a random field $\chi\colon \Omega \times \mathbb{Z}^{d} \to \mathbb{R}^{d}$, a corrector to the process $X$ such that $M_{t} = X_{t} - \chi(\cdot, X_{t})$ is a local martingale under $\Prob_{0}^{\omega}$ for $\mathbb{P}$-a.e.\ $\omega$.  Then, we show that $\mathbb{P}$-a.s.\ the correction part of the random walk $\chi_{t} \equiv \chi(\cdot, X_{t})$ vanishes under the diffusive scaling, while $M_{t}$ follows the QIP.  We approximate our corrector based on a sequence in a Hilbert space  $L^{2}(\mu\mathbb{P})$ (defined below), to which the study of an environment process from the viewpoint of the walker is relevant. More precisely, the stochastic process $X$ induces a continuous-time environment Markov chain $\tau_{X} \omega \equiv (\tau_{X_t}\omega : t \geq 0)$ on $\Omega$. Its generator, $\mathcal{L}$, acts on bounded and measurable functions $\varphi\colon \Omega \to \mathbb{R}$ as
\begin{align}
  \label{eq:def:generator}
  \bigl( \mathcal{L} \varphi \bigr)(\omega)
  \;\ldef\;
  \sum_{z \in \mathcal{N}} c^{\omega}(0, z)\, \bigl( \varphi(\tau_{z} \omega) - \varphi(\omega) \bigr).
\end{align}
The associated Dirichlet form $\mathcal{E}$ acts on bounded and measurable functions $\xi, \varphi\colon \Omega \to \mathbb{R}$ is given by
\begin{align}
  \label{eq:def:dirichlet_form_annealed}
  \mathcal{E}(\xi, \varphi)
  \;\ldef\;
  \mean\bigl[ \xi\, (-\mathcal{L}\varphi) \bigr]. 
\end{align}
In particular, under the stationarity of $\mathbb{P}$ (in Assumption \ref{ass:P}-(i)),
\begin{align}
  \label{eq:prop:dirichlet_norm}
  \mathcal{E}(\varphi, \varphi)
  \;=\;
  \frac{1}{2}\,
  \mean\Biggl[
    \sum_{z \in \mathcal{N}} c_{\mathrm{s}}^{\omega}(0, z)\,
    \bigl(\varphi(\tau_{z} \omega) - \varphi(\omega) \bigr)^{2}
  \Biggr].
\end{align}
Next, we justify that the process $\tau_{X}\omega$ is stationary and ergodic with respect to $\mathbb{P}$, and has non-trivial time evolution.
\begin{prop}
  \label{prop:time-ergodicity-and-non-explosion}
  Suppose Assumption~\ref{ass:P}-(i), \eqref{eqn:prop:doubly_stochastic}, and \eqref{eq:cond:1st-moment} hold. Then, the stochastic processes $X$ and $\tau_X\omega$ are non-explosive, thus well-defined for all $t \geq 0$. Moreover, $\mathbb{P}$ is stationary and ergodic for the induced continuous-time environment  Markov chain $\tau_{X} \omega$, for $\mathbb{P}$-a.e. $\omega$.
\end{prop}
\begin{proof}[Proof outline]
  We define a finite\footnote{$\mu\mathbb{P}$ is a finite measure since $\int_{\omega \in \Omega} \dd \mu\mathbb{P}= \mathbb{E}[\mu^\omega(0)] <\infty.$} measure  $\mu\mathbb{P}$ by $\dd (\mu\mathbb{P})(\omega)\ldef \mu^\omega(0)\dd \mathbb{P}(\omega).$ Let $(\nu_n)_{n\in \mathbb{N}}$ denote the sequence of jumping times of $X$, and $\nu_0:=0.$ It follows easily from the assumptions, that $\mu\mathbb{P}$ is stationary and ergodic for the embedded discrete environment Markov chain $\tau_Y\omega\equiv (\tau_{Y_n}\omega: n\in \mathbb{N}_0)$, where $Y_n^\omega:= X_{\nu_n}^\omega$. The time stationarity is due to the spacial stationarity and doubly stochasticity, since
  \begin{align*}
    \mathbb{E}\bigl[\mathcal{L}\phi\bigr]
    &\;=\;
    \mathbb{E}\Biggl[
      \sum_{z \in \mathcal{N}} c^{\omega}(0,z)\, \phi(\tau_{z}\omega)
    \Biggr]
    - \mathbb{E}\bigl[ \mu^{\omega}(0)\, \phi(\omega) \bigr]
    \\
    &\;=\;
    \mathbb{E}\Biggl[
      \sum_{z \in \mathcal{N}} c^{\tau_{-z}\omega}(0,z)\, \phi(\omega)
    \Biggr]
    - \mathbb{E}\bigl[ \mu^{\omega}(0)\, \phi(\omega) \bigr]
    \;=\;
    0,
  \end{align*}
  for all $\phi$ bounded and measurable. Moreover, for the time ergodicity, we have for every invariant set $A \in \mathcal{F}$, $\mathcal{L}\mathbbm{1}_{A} = 0$, hence
  \begin{align*}
    0
    \;=\;
    \mathbb{E}\bigl[ \mathbbm{1}_{A}(-\mathcal{L})\mathbbm{1}_{A} \bigr]
    \;=\;
    \frac{1}{2}\,
    \mathbbm{E}\Biggl[
      \sum_{z \in \mathcal{N}}c^{\omega}_{s}(0,z)\,
      (\mathbbm{1}_{A} \circ \tau_{z} - \mathbbm{1}_{A})^{2}
    \Biggr].
  \end{align*}
  It then follows that $\mathbbm{1}_{A} \circ \tau_{x} = \mathbbm{1}_{A}$ for all $x \in \mathbb{Z}^{d}$. Further, by Birkhoff's ergodic theorem, $\lim_{n \to \infty} \frac{1}{n} \sum_{k=0}^{n-1} (\mu^{\tau_{Y_n}\omega}(0))^{-1} = \mathbb{E}[\mu^\omega(0)]^{-1} < \infty$, and by Theorem2.33 in \cite[p.97]{Lig10}, this is equivalent to the non-explosiveness for $X$ (thus $\tau_X\omega$), i.e. $\lim_{k~\to \infty}\nu_{k} = \infty$, $\mathbb{P} \otimes \Prob^{\omega}_{0}$ almost surely. The non-explosiveness implies that $\tau_{X} \omega$ as well as $X$ is well-defined for all $t$.  The stationary and ergodic properties of $\mathbb{P}$ for the process $\tau_{X} \omega$ follow by the same arguments used for the discrete chain. 
\end{proof}
For later use, we denote by $L^{2}(\mu\mathbb{P}), L^{2}(\mathbb{P})$ the space of square-integrable functions on $(\Omega, \mathcal{F})$ with respect to the finite measure $\mu\mathbb{P}$ and the probability measure $\mathbb{P}$. Further let $\inp{\cdot }{\cdot }_{L^2(\mu\mathbb{P})}, \inp{\cdot }{\cdot }_{L^2(\mathbb{P})}$ be the corresponding inner products.

\subsection{Corrector construction and harmonic coordinates}
To begin, we introduce some preliminaries before presenting the construction result in Lemma~\ref{lem:corrector_construction}.

\begin{definition}[Cycle condition] 
  \label{def: cycle conditon}
  A random vector $\Psi\colon \Omega \times \mathcal{N} \to \mathbb{R}$, satisfies the \emph{cycle condition} if for any sequence $(x_{0}, \cdots, x_{n})$ in $\mathbb{Z}^{d}$, with $x_{0} = x_{n}$ and $(x_{j-1}, x_{j}) \in \vec{E}_d$, $1 \leq j \leq n$, it holds that
  \begin{align}
    \label{eq:def:cycle_condition}
    \sum_{j=1}^{n}\Psi(\tau_{x_{j-1}}\omega, x_{j} - x_{j-1}) \;=\; 0.
  \end{align}
\end{definition}
\begin{definition}[Cocycle property]
  \label{def: cocycle property}
  A random field $\Psi\colon \Omega \times \mathbb{Z}^{d} \to \mathbb{R}$ satisfies the \emph{cocycle property} if for $\mathbb{P}$-a.e.\ $\omega$
  \begin{align}
    \label{eq:def:cocycle_property}
    \Psi(\omega, y) - \Psi(\omega,x)
    \;=\;
    \Psi(\tau_{x}\omega, y-x),
    \qquad \forall\, x, y \in \mathbb{Z}^{d}.
  \end{align}
\end{definition}
Denote by $\mathcal{G}$ the set of random vectors $\Omega \times \mathcal{N}\ to \mathbb{R}$ that satisfy the cycle condition. Let
\begin{align*}
  \Norm{\Psi}{\mathrm{cov}}^2
  \;\ldef\;
  \frac{1}{2}\,
  \mathbb{E}\Biggl[
    \sum_{z \in \mathcal{N}} c_{\mathrm{s}}^{\omega}(0, z)\, \Psi(\omega, z)^{2}
  \Biggr],
  \qquad \Psi \in \mathcal{G}, 
\end{align*}
and $L^{2}_{\mathrm{cov}} \ldef \{\Psi \in \mathcal{G} : \Norm{\Psi}{\mathrm{cov}} < \infty\}$.
\begin{remark}
  For $\Psi \in \mathcal{G}$, by the cycle condition, $\Psi(\omega, z) = \Psi(\tau_{z}\omega, -z)$, for $(\omega, z) \in \Omega \times \mathcal{N}$. It follows from the stationarity in Assumption~\ref{ass:P}-(i), that
  \begin{align*}
    \Norm{\Psi}{\mathrm{cov}}^2
    \;=\;
    \frac{1}{2}\,
    \mathbb{E}\Biggl[
      \sum_{z \in \mathcal{N}} c^{\omega}(0, z)\, \Psi(\omega, z)^{2}
    \Biggr].
  \end{align*}
  Furthermore, under the finite cycle decomposition \eqref{eq:def:random_edge_weights}, we have
  \begin{align}
    \Norm{\Psi}{\mathrm{cov}}^{2}
    &\;=\;
    \frac{1}{2}\,
    \mathbb{E}\Biggl[
      \sum_{\vec{\gamma}\in \Gamma_0}\sum_{z\in \mathcal{N}}\sum_{x\in \mathbb{Z}^d}\omega_{\vec{\gamma}}(x)\mathbbm{1}_{(0,z)\in x+ \vec{\gamma}}\, \Psi(\omega, z)^2
    \Biggr]
    \nonumber \\
    &\;=\;
    \frac{1}{2}\,
    \mathbb{E}\Biggl[
      \sum_{\vec{\gamma} \in \Gamma_{0}}
      \sum_{(x,y) \in \vec{\gamma}} \omega_{\vec{\gamma}}(0)\, \Psi(\tau_{x}\omega, y-x)^{2}
    \Biggr].
    \label{eq:cov_norm_cycle}
  \end{align}
\end{remark}
It is easy to check the following result:
\begin{prop} 
  \label{prop:cocycle_extension}
  \begin{enumerate}[(i)]
  \item
    Every $u \in \mathcal{G}$ has a unique extension to a random field that satisfies the cocycle property.
  \item
    Suppose $\mathbb{P}[0 < c_{\mathrm{s}}^{\omega}(0,z) < \infty] = 1$, for all $z \in \mathcal{N}$. Then, $(L^{2}_{\mathrm{cov}}, \Norm{\cdot}{\mathrm{cov}})$ is a Hilbert space. 
  \end{enumerate}
\end{prop}
Denote by $D = (D_{z})_{z \in \mathcal{N}}$ the \emph{horizontal derivative} operator on measurable functions $\varphi\colon \Omega \to \mathbb{R}$, with
\begin{align*}
  D_{z}\varphi(\omega) \;\ldef\; \varphi(\tau_z\omega) - \varphi(\omega).
\end{align*}
We write $(D\varphi)(\omega, z) = D_{z}\varphi(\omega)$. Let $\mathbb{B}$ be the set of bounded and measurable functions on $(\Omega, \mathcal{F})$, and write $D\mathbb{B} = \{D\varphi : \varphi \in \mathbb{B}\}$.  Denote by $L^{2}_{\mathrm{pot}}$ the closure of $D\mathbb{B}$ in $L^{2}_{\mathrm{cov}}$, and $L^{2}_{\mathrm{sol}}$ the orthogonal complement of $L^{2}_{\mathrm{pot}}$.

For each $i = 1, \dots, d$, define a random field $\Pi^{i}\colon \Omega \times \mathbb{Z}^{d} \to \mathbb{R}, (\omega, x) \mapsto x^{i}$. We call $\Pi^{i}$ the \emph{position field} in the $i$-th coordinate. Further, let
\begin{align}
  \label{eq:def:local_drift}
  V^{i}(\omega, x)
  \;\ldef\;
  L^{\omega}\Pi^{i}(\omega, x)
  \;=\;
  \sum_{z \in \mathcal{N}} c^{\omega}(x, x+z) z^{i},
\end{align}
with $L^{\omega}$ as defined in \eqref{eq:def:generator_quenched}, $z^{i}$ denoting the $i$-th coordinate of $z$. $V^{i}$ is called the \emph{local drift} in the $i$-th coordinate. 

\begin{remark}
  \label{rm:local_drift}
  In view of \eqref{prop:shifted-rate}, $V^{i}(\omega, x) = V^{i}(\tau_{x} \omega, 0)$, for any $i = 1, \dots, d$ and $x \in \mathbb{Z}^{d}$.
\end{remark}

We state a sufficient condition that ensures the local drift $V^i$ has a finite $\mathcal{H}_{-1}$ norm\footnote{
  The $\mathcal{H}_{-1}$ norm is with respect to a pre-Hilbert space $\mathcal{H} \subset L^{2}(\mu \mathbb{P})$, with $\inp{\xi}{\varphi}_H \ldef \mathbb{E}[\xi ( -\mathcal{L}_{\mathrm{s}}\varphi)]$, with $\mathcal{L}_{\mathrm{s}}= \frac{1}{2}(\mathcal{L} + \mathcal{L}^{\ast})$ and $\mathcal{L}^{\ast}$ being the adjoint of $\mathcal{L}$. It follows, $\inp{\xi}{\xi}_H = \norm{D\xi}^{2}_{\mathrm{cov}}$.
}.
\begin{assumption}
  \label{ass:finite_H_-1_norm}
  Assume that
  \begin{align}
    \label{cond:finite_H_-1_norm}
    \alpha
    \;\ldef\;
    \max_{i = 1, \dots, d}
    \sqrt{
      \mathbb{E}\Biggl[
        \sum_{\vec{\gamma} \in \Gamma_{0}} \omega_{\vec{\gamma}}(0)\, \sum_{x \in \vec{\gamma}}(x^{i})^{2}
      \Biggr]
    }
    \;<\;
    \infty, 
  \end{align}
  where $x^i$ denotes the $i$-th coordinate of $x$. 
\end{assumption}
\begin{remark} 
  Assumption~\ref{ass:finite_H_-1_norm} is satisfied when $\mathbb{E}\bigl[ \sum_{\vec{\gamma} \in \Gamma_{0}} \omega_{\vec{\gamma}}(0) \abs{\vec{\gamma}}^{3}\bigr] < \infty$, which is covered by Assumption~\ref{ass:p-q-moment}-(i) with $p = 1$.
\end{remark}
\begin{remark}
  According to \cite{KT17}, in the doubly stochastic case (\eqref{eqn:prop:doubly_stochastic} is satisfied),  if  $c_{\mathrm{s}}^{\omega}$ is strong-elliptic, then $\mathcal{H}_{-1}$ condition is equivalent to the existence of $L^{2}$-integrable stream tensors. The stream tensors can be interpreted as special length-four cycles that allow negative cycle weights. 
\end{remark}
The main result of this subsection is as follows:
\begin{lemma}[Corrector construction]
  \label{lem:corrector_construction}
  Suppose Assumption~\ref{ass:P} and \ref{ass:finite_H_-1_norm} hold. Then, for each $i = 1, \dots, d$,
  \begin{enumerate}[(i)]
  \item
    there exists a random field $\chi^{i}\colon \Omega \times \mathbb{Z}^{d} \to \mathbb{R}$ that satisfies the cocycle property such that its restriction $\chi^{i}\big|_{\Omega \times \mathcal{N}} \in L^{2}_{\mathrm{pot}}$, and, for $\mathbb{P}$-a.e.\ $\omega$ and every $x \in \mathbb{Z}^{d}$,
    \begin{align}
      \label{eq:corrector_construction}
      L^{\omega} \chi^{i}(\omega, x) \;=\; V^{i}(\omega, x). 
    \end{align}
    
  \item
    $\Phi^{i} \ldef \Pi^{i} - \chi^{i}$ satisfies $\Phi^{i}\big|_{\Omega \times \mathcal{N}} \in L^{2}_{\mathrm{cov}}$, and, for $\mathbb{P}$-a.e.\ $\omega$ and every $x \in \mathbb{Z}^d$,
    \begin{align}
      \label{eqn:harmonic-coordinates}
      L^{\omega}\Phi^{i}(\omega, x) \;=\; 0.
    \end{align}
  \end{enumerate}
\end{lemma}
By the definition of $V^{i}$, and that $\Pi^{i}\big|_{\Omega \times \mathcal{N}} \in L^{2}_{\mathrm{cov}}$, (ii) follows immediately from (i). We refer to $\chi =  (\chi^{i})_{i = 1, \dots, d}$ as a \emph{corrector}, and $\Phi =  (\Phi^{i})_{i = 1, \dots, d}$ as \emph{harmonic coordinates}.
\begin{remark}
  If $L^{\omega}$ is symmetric, i.e.\ $c^{\omega} = c_{\mathrm{s}}^{\omega}$, then $\chi^{i}\big|_{\Omega\ \times \mathcal{N}}$, $\Phi^{i}\big|_{\Omega \times \mathcal{N}}$ are the unique orthogonal projections of $\Pi^{i}$ respectively on $L^{2}_{\mathrm{pot}}$, $L^{2}_{\mathrm{sol}}$, e.g.\ see \cite[Section 2.1]{ADS15}.
\end{remark}
For the proof of Lemma~\ref{lem:corrector_construction}, we need following estimates.
\begin{prop}[Weak sector condition and testable $\mathcal{H}_{-1}$ condition]
  \label{prop:weak_sec_H_-1}
  Under the same assumption of Lemma \ref{lem:corrector_construction}, 
  \begin{enumerate}[(i)]
  \item
    The domain the Dirichlet form given in \eqref{eq:def:dirichlet_form_annealed} can be extended to $L^{2}(\mu \mathbb{P}) \times L^{2}(\mu \mathbb{P})$, and satisfies a \emph{weak sector condition}: there exists a constant $C_{\mathrm{sec_w}} > 0$ such that
    \begin{align}
      \label{eqn:weak-sector}
      \abs{\mathcal{E}(\xi, \varphi)}^{2}
      \;\leq\;
      C_{\mathrm{sec_w}} \norm{\xi}^{2}_{L^2(\mu\mathbb{P})}\, \norm{\varphi}^{2}_{L^2(\mu\mathbb{P})},
      \qquad \forall\, \xi, \varphi \in L^{2}(\mu\mathbb{P});
    \end{align}
    
  \item
    for $i = 1, \dots, d$, and $V^{i}$, $\alpha$ given in \eqref{eq:def:local_drift}  and Assumption~\ref{ass:finite_H_-1_norm},
    \begin{align}
      \label{eqn:H_-1}   
      \abs{\mathbb{E}[\xi V^i(\cdot, 0)]}
      \;\leq\;
      \sqrt{2} \alpha\, \norm{D\xi}_{\mathrm{cov}},
      \qquad \forall\, \xi \in L^2(\mu\mathbb{P}).
    \end{align}
  \end{enumerate}
\end{prop}

\begin{proof}
  (i) is clear. First notice that the set of measurable and bounded functions $\mathbb{B}$ is dense in $L^{2}(\mu \mathbb{P})$, and for every $\xi, \varphi \in \mathbb{B}\times \mathbb{B}$, we may write
  \begin{align*}
    \mathcal{E}(\xi, \varphi)
    \;=\;
    \mathbb{E}\bigl[\xi(-\mathcal{L}\varphi)\bigr]
    \;=\;
    \inp{\xi}{-\mu(0)^{-1} \mathcal{L}\varphi}_{L^2(\mu\mathbb{P})}.
  \end{align*}
  The RHS makes sense, since by Jensen's inequality and the stationarity of $\mathbb{P}$,
  \begin{align}
    \label{eq:prop:weak_sector}
    \norm{-\mu(0)^{-1}\mathcal{L}\varphi}^{2}_{L^{2}(\mu\mathbb{P})}
    \;\leq\;
    \mathbb{E}\Biggl[
      \sum_{z \in \mathcal{N}} c^{\omega}(0, z)\, (D_{z}\varphi(\omega))^{2}
    \Biggr]
    \;=\;
    2\, \norm{D\varphi}^{2}_{\mathrm{cov}}
    \;\leq\;
    4\, \norm{\varphi}^{2}_{L^{2}(\mu\mathbb{P})}.
  \end{align}
  Now we show (ii).  It follows from \eqref{eq:def:random_edge_weights}, for $\omega \in \Omega$,
  \begin{align*}
    V^i(\omega, 0)
    &\;=\;
    \sum_{\vec{\gamma} \in \Gamma_{0}} \sum_{z \in \mathcal{N}} \sum_{x \in \mathbb{Z}^{d}} \omega_{\vec{\gamma}}(x)\, \mathbbm{1}_{(0, z) \in x + \vec{\gamma}}\, z^{i}
    \\
    &\;=\;
    \sum_{\vec{\gamma} \in \Gamma_{0}} \sum_{z \in \mathcal{N}} \sum_{x \in \mathbb{Z}^{d}} \omega_{\vec{\gamma}}(x)\, \mathbbm{1}_{(-x, z-x) \in \vec{\gamma}}\,
    \bigl( (z-x)^{i} - (-x)^{i} \bigr)
    \\
    &\;=\;
    \sum_{\vec{\gamma} \in \Gamma_{0}} \sum_{(x, y) \in \vec{\gamma}} \omega_{\vec{\gamma}}(-x) \bigl(y^{i} - x^{i}\bigr).
  \end{align*}
  Hence, by stationarity, Cauchy-Schwarz inequality and \eqref{eq:cov_norm_cycle},
  \begin{align*}
    \abs{\mathbb{E}[\xi V^i(\cdot, 0)}
    &\;=\;
    \Biggl|
      \sum_{\vec{\gamma} \in \Gamma_{0}} \sum_{(x, y) \in \vec{\gamma}}
      \mathbb{E}\bigl[ \xi\, \omega_{\vec{\gamma}}(-x) \bigr] \bigl(y^{i} - x^{i}\bigr)
    \Biggr|
    \\
    &\;=\;
    \Biggl|
      \sum_{\vec{\gamma} \in \Gamma_{0}} \sum_{(x, y) \in \vec{\gamma}}
      \mathbb{E}\bigl[
        (\xi \circ \tau_{x} - \xi \circ \tau_{y})\, \omega_{\vec{\gamma}}(0)
      \bigr]\, x^{i}
    \Biggr|
    \\
    &\;\leq\;
    \mathbb{E}\Biggl[
      \sum_{\vec{\gamma} \in \Gamma_{0}} \omega_{\vec{\gamma}}(0)
      \sum_{x \in \vec{\gamma}}(x^{i})^{2}
    \Biggr]^{1/2}\,
    \mathbb{E}\Biggl[
      \sum_{\vec{\gamma} \in \Gamma_{0}} \sum_{(x, y) \in \vec{\gamma}} \omega_{\vec{\gamma}}(0)\, (\xi \circ \tau_y - \xi \circ \tau_x )^2
    \Biggr]^{1/2}
    \\
    &\;=\;
    \sqrt{2} \alpha\, \norm{D\xi}_{\mathrm{cov}}.
  \end{align*}
\end{proof}
\begin{remark} 
  \label{rem:subset_of_L^2pot}
  $\{D\varphi : \varphi \in L^{2}(\mu\mathbb{P})\} \subset L^{2}_{\mathrm{pot}}$. Indeed, since for every $\varphi \in L^{2}(\mu\mathbb{P})$, there exists a sequence of bounded functions $(\varphi_{n}) \subset \mathbb{B}$, such that $\norm{\varphi_{n} -\varphi}_{L^{2}(\mu\mathbb{P})} \xrightarrow[]{n \to \infty} 0$. It follows $\norm{D\varphi_{n} - D\varphi}_{\mathrm{cov}} \leq \sqrt{2} \norm{\varphi_{n} -\varphi}_{L^{2}(\mu\mathbb{P})} \xrightarrow[]{n \to \infty} 0$, hence $D\varphi \in \overline{D\mathbb{B}}^{\norm{\cdot}_\mathrm{cov}} = L^{2}_{\mathrm{pot}}$.
\end{remark}

\begin{remark}
  Due to the crude estimation in the final part of the proof, we did not expect $\sqrt{2}\alpha$ to be a sharp $\mathcal{H}_{-1}$ norm for $V^{i}(\cdot, 0)$.
\end{remark}

We are ready to show Lemma \ref{lem:corrector_construction}.
\begin{proof}[Proof of Lemma \ref{lem:corrector_construction}]
  It is sufficient to show (i). Further, considering Remark~\ref{rm:local_drift} and Proposition \ref{prop:cocycle_extension}-(i), it is enough to show \eqref{eq:corrector_construction} for $x = 0$.

  Fix $i \in \{1, \dots, d\}$, and for $\xi, \varphi \in L^{2}(\mu\mathbb{P}), \lambda \in (0, 1)$, let
  \begin{align*}
    B_{\lambda}(\xi, \varphi)
    &\;\ldef\;
    \mathcal{E}(\xi, \varphi) + \lambda\, \inp{\xi}{\varphi}_{L^{2}(\mu\mathbb{P})},
    \\
    Q^{i}(\xi)
    &\;\ldef\;
    -\mathbb{E}\bigl[\xi V^{i}(\cdot, 0)\bigr].
  \end{align*}
  It follows from \eqref{eq:def:dirichlet_form_annealed}, \eqref{eq:prop:dirichlet_norm}, \eqref{eq:prop:weak_sector}, and Proposition \ref{prop:weak_sec_H_-1}, that  $(B_{\lambda})_{\lambda \in (0, 1)}$ is a set of bounded, coercive bilinear form on $L^{2}(\mu\mathbb{P}) \times L^{2}(\mu\mathbb{P}) $ and that $Q^{i}$ is a bounded linear operator on $L^{2}(\mu\mathbb{P})$. Thus, by the Lax-Milgram theorem, for each $\lambda \in (0, 1)$, there exists $\varphi_{\lambda}^{i} \in L^{2}(\mu\mathbb{P})$, such that
  \begin{align}
    \label{eq:prop:weak_poisson}
    B_{\lambda}(\xi, \varphi^{i}_{\lambda}) \;=\; Q^{i}(\xi),
    \qquad \forall\, \xi \in L^{2}(\mu\mathbb{P}). 
  \end{align}
  Furthermore, due to Remark \ref{rem:subset_of_L^2pot} and
  \begin{align*}
    \norm{D\varphi^{i}_{\lambda}}^{2}_{\mathrm{cov}}
    + \lambda\, \norm{\varphi^{i}_{\lambda}}^{2}_{L^{2}(\mu\mathbb{P)}}
    \;\leq\;
    \sqrt{2}\alpha\, \norm{D\varphi^{i}_{\lambda}}_{\mathrm{cov}},
  \end{align*}
  $(\varphi^{i}_{\lambda})_{\lambda \in (0, 1)}$ is uniformly bounded in $L^{2}_{\mathrm{pot}}$, with
  \begin{align}
    \label{eqn:prop:norm_bounds}
    \norm{D\varphi^{i}_{\lambda}}_{\mathrm{cov}} \;\leq\; \sqrt{2}\alpha,
    \qquad \text{and} \qquad
    \lambda \norm{\varphi^{i}_{\lambda}}_{L^{2}(\mu\mathbb{P)}}
    \;\leq\;
    \sqrt{2 \lambda} \alpha.
  \end{align}
  Hence, there exists a sequence $\lambda_{n} \xrightarrow{n \to \infty}0$, such that $D\varphi_{\lambda_{n}}^{i}$ converges weakly to $\chi^{i}$ in $L^{2}_{\mathrm{pot}}$. By Mazur's lemma \cite[Lemma 10.1]{RR04}, there exists a function $N\colon \mathbb{N} \to \mathbb{N}$, and a sequence of sets of non-negative real numbers $(a_{n,k})_{n\leq k \leq N(n)}$, with $\sum_{k=n}^{N(n)} a_{n, k} = 1$, such that the sequence $(D\varphi^{i,n})$, with $\varphi^{i,n} = \sum_{k=n}^{N(n)} a_{n,k} \varphi^{i}_{\lambda_k}$ converges strongly to $\chi^{i}$ in $L^{2}_{\mathrm{pot}}$. In other words,
  \begin{align}
    \label{prop:strong_convergence_corrector}
    \norm{D\varphi^{i,n} - \chi^{i}}_{\mathrm{cov}} \;\xrightarrow{n\to \infty}\; 0.
  \end{align}
  Next, we show that $\mathbb{P}$-a.s.\ \eqref{eq:corrector_construction} holds. It suffices\footnote{
    To see this, write $\mathbb{E}\big[\xi \sum_{z \in \mathcal{N}} c^{(\cdot)}(0,z) \chi^{i}(\cdot, z) \big] = \inp{\xi}{\mu(0)^{-1} \sum_{z \in \mathcal{N}} c^{(\cdot)}(0,z) \chi^{i}(\cdot, z)}_{L^{2}(\mu\mathbb{P})}$, and $\mathbb{E}[\xi V^{i}(\cdot, 0)] = \inp{\xi}{\mu(0)^{-1} V^{i}(\cdot, 0)}_{L^{2}(\mu \mathbb{P})}$. They are justified by a similar argument to \eqref{eq:prop:weak_sector}.
  } to show
  \begin{align}
    \mathbb{E}\Biggl[
      \xi \sum_{z \in \mathcal{N}}c^{\omega}(0, z)\, \chi^{i}(\omega, z)
    \Biggr]
    \;=\;
    \mathbb{E}\bigl[ \xi V^{i}(\cdot, 0) \bigr],
    \qquad \forall\, \xi \in \mathbb{B}.
  \end{align}
  Notice that for each $\xi \in \mathbb{B}$,
  \begin{align}
    \label{eqn:prop:weak_poisson_strong_seq}
    \mathcal{E}(\xi, \varphi^{i, n})
    \;=\;
    \sum_{k=n}^{N(n)} a_{n,k}\, \mathcal{E}(\xi, \varphi^{i}_{\lambda_{n}})
    \;=\;
    Q^{i}(\xi) - \sum_{k=n}^{N(n)} a_{n,k} \lambda_{k} \inp{\xi}{\varphi^{i}_{\lambda_{k}}}_{L^{2}(\mu\mathbb{P})}.
  \end{align}
  By the definition \eqref{eq:def:dirichlet_form_annealed}, $\mathcal{E}(\xi, \varphi^{i,n}) = \mathbb{E}\bigl[ \xi(-\sum_{z \in \mathcal{N}} c^{\omega}(0, z) D\varphi^{i,n})\bigr]$. Moreover, by Cauchy-Schwarz's inequality,
  \begin{align*}
    \Biggl|
      \mathbb{E}\Biggl[
        \xi \biggl(-\sum_{z \in \mathcal{N}} c^{\omega}(0, z)\, D_{z}\varphi^{i,n} \biggr)
      \Biggr]
    \Biggr|
    \;\leq\;
    \norm{\xi}_{L^{\infty}}\,
    \mathbb{E}\bigl[\abs{\mathcal{L}\varphi^{i, n}}\bigr]
    \;\leq\;
    \sqrt{2}\, \norm{\xi}_{L^{\infty}}\,
    \mathbb{E}[\mu]^{1/2}\, \norm{D\varphi^{i, n}}_{\mathrm{cov}}.
  \end{align*}
  On the other hand, by \eqref{eqn:prop:norm_bounds},
  \begin{align*}
    \abs{\lambda_{k}\inp{\xi}{\varphi^{i}_{\lambda_k}}_{L^{2}(\mu\mathbb{P})}}
    \;\leq\;
    \norm{\xi}_{L^{\infty}}\, \lambda_{k} \norm{\varphi^{i}_{\lambda_{k}}}_{L^{2}(\mu\mathbb{P})}
    \;\leq\;
    \norm{\xi}_{L^{\infty}} \sqrt{2\lambda_{k}} \alpha, 
  \end{align*}
  and it follows that $\abs{\sum_{k=n}^{N(n)} a_{n,k} \lambda_{k} \inp{\xi}{\varphi^{i}_{\lambda_k}}_{L^{2}(\mu\mathbb{P})}} \leq \norm{\xi}_{L^{\infty}} \sqrt{2\lambda_{n}}\alpha$.

  Thus, as $n \to \infty$,  by the strong convergence \eqref{prop:strong_convergence_corrector}, the LHS of \eqref{eqn:prop:weak_poisson_strong_seq} converges to  $\mathbb{E}\bigl[\xi \sum_{z \in \mathcal{N}} c^{\omega}(0, z) \chi^{i}(\cdot, z)\bigr]$, with its LHS converging to $Q^{i}(\xi) = \mathbb{E}\bigl[\xi V^{i}(\cdot, 0)\bigr]$. At this point, we conclude Lemma \ref{lem:corrector_construction}.
\end{proof}

\subsection{Properties of the corrector and local martingale}
For $x\in \mathbb{Z}^d, r>0$, let $B(x_0, r):= \{x\in \mathbb{Z}^d: \norm{x-x_0}_{\infty} \leq r\}$, and $\partial B(x_0, r):=\{x\in \mathbb{Z}^d: \norm{x-x_0}_{\infty} = \floor{r}\}$, with $\norm{x}_{\infty}= \max_{i=1, \dots, d}\abs{x_i}.$  We write $B(n)\equiv B(0, n)$ for $n\in \mathbb{N}$. Now we introduce \textit{space-averaged norms} on function $u: B\to \mathbb{R}$, with $B$ being a non-empty set of $\mathbb{Z}^d$: for $p\in (0, \infty)$, define 
\begin{align*}
    \norm{u}_{p, B}:= \bigg(\frac{1}{\abs{B}} \sum_{x\in B} 
    \abs{u(x)}^p\bigg)^{1/p}, 
\end{align*}
and $\norm{u}_{\infty, B}:= \sup_{x\in B}\abs{u(x)}.$


A crucial step for the proof of Theorem \ref{thm:QFCLT} is to manage the corrector on a large scale, such that the maximal correction inside a box is negligible compared to the size of the box. This property is referred to as the \textit{sublinearity of the corrector}. Two results follows: (1) the quenched invariance principle holds for the local martingale process $\Phi(\omega, X_t)$; (2) $\chi(\omega, X_t)$ vanishes $\mathbb{P}\times \Prob_0^\omega$-almost surely under the diffusive scaling. 

We now formalise the above results: 

\begin{lemma}[Corrector sublinearity and local Martingale QFCLT] \label{lem:sublinearCorrector_martingaleFCLT} Suppose $d\geq 2$, Assumption ~\ref{ass:P}-(i) and Assumption ~\ref{ass:p-q-moment} hold. Let $\chi, \Phi$ be as constructed in Lemma \ref{lem:corrector_construction}. Then 
\begin{enumerate}[(i)]
    \item $\chi$ is sublinear on a large scale, i.e. for $\mathbb{P}$-a.e. $\omega$,  and all $i=1, \dots, d,$
    \begin{align}
\label{eqn:sublinearity}
    \lim_{n\to\infty}\frac{1}{n}\norm{\chi^i(\omega, \cdot)}_{\infty, B(n)}=0; 
    \end{align}
    \item  $M:= (\Phi(\cdot, X_t))_{t\geq 0}$ satisfies the quenched invariance principle, i.e. QFCLT holds for $M$ with a deterministic nondegenerate covariance matrix $\Sigma^2$ given by 
\begin{align}
\label{def:covariance}
    \Sigma^2_{i,j}=\mathbb{E}\Big[\sum_{z\in \mathcal{N}} c_{\mathrm{s}}^{\omega}(0, z)\Phi^i(\omega, z) \Phi^j(\omega, z) \Big],
\end{align}
for $1\leq i, j \leq d;$
\item $\chi$ vanishes along random walks under the diffusive scaling, i.e., 
for $\mathbb{P}$-a.e. $\omega$, and all $i=1, \dots, d$, $T>0$, 
\begin{align}
\label{eqn:vanishing-corrector}
    \lim_{n\to \infty}\sup_{t\in [0, T]} \frac{1}{n}\abs{\chi^i(\omega, X_{n^2t})}=0 \qquad \text{in } \Prob_0^\omega.
\end{align}
\end{enumerate}
\end{lemma}
\begin{remark}
    Lemma \ref{lem:sublinearCorrector_martingaleFCLT}-(ii) follows from the construction of $\chi$ and $\Phi$, as well as the space-averaged $\ell^1$-sublinearity of $\chi$. Thus, it holds under Assumption ~\ref{ass:P}-(i) and Assumption ~\ref{ass:p-q-moment} with $p=q=1$. 
\end{remark}

We divide the proof in  three parts.

\subsubsection{Sublinearity of the corrector}

To show Lemma \ref{lem:sublinearCorrector_martingaleFCLT}-(i), we need to control the maximum of (locally) harmonic function, in particular $\Phi^i$, by its space-averaged norm, which will be governed by the space-averaged sublinearity of $\chi^i$, with an error term that is at most linear. 
\begin{prop}[Maximal inequality for harmonic functions] 
\label{prop:maximum-inequality}
Fix $\omega \in \Omega$. Let $u: \mathbb{Z}^d \to \mathbb{R}$ be either non-negative $L^\omega$-subharmonic $(-L^\omega u \leq 0)$,  or $L^\omega$-harmonic $(L^\omega u = 0)$ on $B(x, n)$, for some $x\in \mathbb{Z}^d$. Then for $p, q\in (1, \infty]$ with $1/p+1/q <2/d$, there exist constants $\kappa = \kappa(d, p, q) \in (1/2, \infty)$ and $C_{\mathrm{Max}}= C_{\mathrm{Max}}(d, p, q)$ such that for all $1/2 \leq \sigma' < \sigma \leq 1$,
\begin{align}
\label{eqn:max-ineq}
    \norm{u}_{\infty, B(x, \sigma'n)} \leq C_{\mathrm{Max}}\bigg( \frac{ \norm{\mu^{(2), \omega}}_{p, B(x, \sigma n)}\norm{\nu^\omega}_{q, B(x, \sigma n)}}{(\sigma - \sigma')^2}\bigg)^\kappa \norm{u}_{2p_*, B(x, \sigma n)},
\end{align}
where $p_*=p/(p-1)$ is the Hölder conjugate of $p$.
\end{prop}
We defer the proof to Section \ref{sec:max_ineq}. 

 \begin{prop}[$(\ell^{2\rho}, B)$-sublinearity] 
\label{prop:space-averaged-sublinearity}
    Let $\rho = \rho(d, q) := d/(d-2+d/q)$. It holds that for $\mathbb{P}$-a.e. $\omega$, and all $i=1, \dots, d$, 
\begin{align}
\label{eqn:space-averaged-sublinearity}
    \lim_{n\to\infty}\frac{1}{n}\norm{\chi^i(\omega, \cdot)}_{2\rho, B(n)}=0. 
\end{align}
\end{prop}
\begin{remark} Since $2\rho >1$, $(\ell^1, B)$-sublinearity of the corrector is an immediate consequence of Proposition \ref{prop:space-averaged-sublinearity}.
\end{remark}

\begin{proof}[Proof of Lemma \ref{lem:sublinearCorrector_martingaleFCLT}-(i)] It can be easily deduced from Proposition \ref{prop:space-averaged-sublinearity} and Proposition \ref{prop:maximum-inequality}, by taking a two-scale argument used in \cite[Proposition 2]{BS20}. The two-scale method is commonly used to transform the problem of controlling the solution to a Poisson equation into one of controlling the solution to a Laplace equation. In this setting, the procedure goes as follows: 1) first subdivide the box $B(n)$ into smaller boxes $B(\floor{\frac{n}{m}}z,\floor{\frac{n}{m}})$, $z\in B(m)$, for $m\in \mathbb{N}$ satisfying $n\geq m(m+1)$ (hence the term \textit{two-scale}); 2) on each sub-box, we have $\abs{\chi^i} \leq \abs{z^i\floor{\frac{n}{m}}-\Phi^i} + \floor{\frac{n}{m}}$, where the first term on the RHS is controlled by the $\norm{z^i\floor{\frac{n}{m}}-\Phi^i}_{2p_*, B(\floor{\frac{n}{m}}z,2\floor{\frac{n}{m}})}\leq \norm{\chi^i}_{2p_*, B(\floor{\frac{n}{m}}z,2\floor{\frac{n}{m}})} + 2\floor{\frac{n}{m}}$ , which is of order $O(m^d)\norm{\chi^i(\omega, \cdot)}_{2p_*, B(3n)}+ O(\frac{1}{m})$ with a prefactor provided in \eqref{eqn:max-ineq}. The prefactor is bounded under the spatial ergodic theorem, and the p-q moment condition in Assumption \ref{ass:p-q-moment}. In view of the $(\ell^{2\rho}, B)$-sublinearity, $\rho>p_*$, and the free choice of $m$, it yields the $(\ell^{\infty}, B)$-sublinearity. 
\end{proof}
\begin{proof}[Proof of Proposition \ref{prop:space-averaged-sublinearity}]
By Lemma \ref{lem:corrector_construction}, $\chi^i\in L^2_{\mathrm{pot}}$, there exists a sequence $(\phi^i_k)_{k\in \mathbb{N}} \subset \mathbb{B}$ such that 
\begin{align}
\label{eq:corrector_conv_seq}
    \lim_{k\to \infty} \norm{\chi^i-D\phi_k^i}_{\mathrm{cov}}=0.
\end{align}

    Write  $(f)_{B(n)}= \frac{1}{\abs{B(n)}}\sum_{x\in B(n)}f(x)$. Fix $i$, and let $f_k(x) = \chi^i(\omega, x) - \phi^i_k(\tau_x \omega)$. 
    First, by the local Poincaré inequality (Proposition \ref{prop:weighted-graph}-(ii)) and the spatial ergodic theorem (cf. \cite[Theorem 2.8 in Chapter 6]{Kre85}), 
    \begin{align*}
         \uplim_{k\to \infty}\uplim_{n\to \infty}\frac{1}{n^2}\norm{f_k - (f_k)_{B(n)}}_{2\rho, B(n)}^2 
         & \leq \uplim_{k\to \infty}\uplim_{n\to \infty} C_{\mathrm{LP}} \norm{\nu^\omega}_{q, B(n)} \frac{1} {\abs{B(n)}}\mathcal{E}^\omega_{B(n)} (f_k)\\ 
         & \leq \lim_{k\to \infty} C_{\mathrm{LP}} \norm{\nu^\omega}_{L^q(\mathbb{P})} \norm{\chi^i-D\phi_k^i}_{\mathrm{cov}}^2 \\ &  \stackrel{\eqref{eq:corrector_conv_seq}}{=}0. 
    \end{align*}
    It follows that, for arbitrary $\delta>0$, there exist $k, n_0 \in \mathbb{N}$, such that for all $n\geq n_0$, \eqref{eq:reg:large_balls} holds,  and 
   $\frac{1}{n}\norm{f_k - (f_k)_{B(n)}}_{2\rho, B(n)}<\delta$. Next, we apply a dyadic iteration scheme similar to \cite[Section 3.4]{FO17}. Let $n_m= 2^mn_0$, $l\in \mathbb{N}$, then for $1\leq l\leq m,$
     \begin{align*}
       \frac{1}{n_m}\abs{(f_k)_{B(n_{l-1})} - (f_k)_{B(n_{l})}} & \leq \frac{1}{n_m}\norm{f_k - (f_k)_{B(n_l)}}_{1, B(n_{l-1})}\\
         &  \leq 2^{l-m}\frac{1}{n_l}\frac{\abs{B(n_l)}}{\abs{B(n_{l-1})}}
         \norm{f_k - (f_k)_{B(n_l)}}_{1, B(n_{l})} \\
         & \stackrel{\eqref{eq:reg:large_balls}}{\leq} \frac{C_{\mathrm{reg}}}{C_{\mathrm{reg}}} 2^{d-m+l} \delta. 
     \end{align*}
     Thus, 
     \begin{align*}
         & \uplim_{m\to \infty}\frac{1}{n_m}\norm{\chi^i(\omega, \cdot)}_{2\rho, B(n_m)} \\
         & \qquad \quad \leq \uplim_{m\to \infty}\frac{1}{n_m}\norm{f_k}_{2\rho, B(n_m)} + \uplim_{m\to \infty}\frac{2}{n_m}\norm{\phi_k^i}_{L^\infty} \\
         & \qquad \quad \leq  \uplim_{m\to \infty} \Big(\frac{1}{n_m}\norm{f_k - (f_k)_{B(n_m)}}_{2\rho, B(n_m)} +\frac{1}{n_m}\sum_{l=1}^m \abs{(f_k)_{B(n_l)} - (f_k)_{B(n_{l-1})}}  \Big)\\
         & \qquad \qquad \qquad + \uplim_{n\to \infty} \frac{1}{n_m}\abs{(f_k)_{B(n_0)}} +  \uplim_{m\to \infty}\frac{2}{n_m}\norm{\phi_k^i}_{L^\infty} \\
         & \qquad \quad \leq \uplim_{m\to \infty}(\delta +  \frac{C_{\mathrm{reg}}}{C_{\mathrm{reg}}} \delta \sum_{l=1}^{m}2^{d-m+l}) + 0 + 0\\
         &  \qquad \quad \leq (1+ \frac{C_{\mathrm{reg}}}{C_{\mathrm{reg}}} 2^{d+1})\delta.  
     \end{align*}
      Since $\delta$ is chosen arbitrarily,  
      \begin{align*}
          \uplim_{k\to \infty}\uplim_{n\to \infty}\frac{1}{n_m}\norm{\chi^i(\omega, \cdot)}_{2\rho, B(n_m)}=0.
      \end{align*}
      On the other hand, by \eqref{eq:reg:large_balls}, there exists a constant $C$, such that  for every $n\in [n_{m-1}, n_m-1]$, $m\in \mathbb{N}$, 
      \begin{align*}
          \frac{1}{n}\norm{\chi^i(\omega, \cdot)}_{2\rho, B(n)} \leq  C \frac{1}{n_m}\norm{\chi^i(\omega, \cdot)}_{2\rho, B(n_m)}.
      \end{align*}
      This concludes \eqref{eqn:space-averaged-sublinearity}.
\end{proof}

\subsubsection{Local martingale QFCLT}
For $\omega \in \Omega$, $i=1, \dots, d$, $n\in \mathbb{N}$,  let $M^{(n), i, \omega}_t:= \frac{1}{n}\Phi^i(\omega, X_{n^2t}), $ and $M_t^{i, \omega}\equiv M_t^{(1), i, \omega}.$

Form Lemma \ref{lem:corrector_construction}-(ii), it follows easily that $M^{i, \omega}$ is a locally $L^2(\Prob_0^\omega)$-integrable local martingale, with localizing sequence $(\nu_k)_{k\in \mathbb{N}}$  of jumping times of $X$. Moreover, the quadratic variation $\sigma_t^{i, \omega}:= [M^{i, \omega}]_t=\sum_{0\leq s\leq t}\abs{\Delta M_s^{i, \omega}}^2 =\sum_{k=1}^{N(t)}\abs{\Delta M_{\nu_k}^{i, \omega}}^2$ is locally $L^1(\Prob_0^\omega)$-integrable, with its unique \textit{compensator} constructed\footnote{E.g. see L\'{e}vy system theorem \cite[Theorem VI.28.1]{RW20}. \label{ft:Levy-system-theorem}} as
\begin{align*}
    \tilde{\sigma}^{i, \omega}_t :=  \langle\,M^{i, \omega}\,\rangle_t
    &= \sum_{k=1}^{N(t^-)}\Mean_0^\omega[\abs{\Delta M_{\nu_k}^{i, \omega}}^2 \mid \mathcal{F}^\omega_{\nu_{k-1}}] \, \mu^\omega(\tau_{X_{\nu_k}})(t-\nu_k) 
   \\ & = \int_0^t \sum_{z\in \mathcal{N}} c^\omega(\tau_{X_s}\omega, z)\abs{\Phi^i(\tau_{X_s}\omega, z)}^2\dd s.
\end{align*}
The formulation of the compensator is useful in the next proof. 

\begin{proof}[Proof of Lemma \ref{lem:sublinearCorrector_martingaleFCLT}-(ii)] We use Helland's local martingale FCLT to prove the weak convergence, then show that the space-averaged $\ell^1$-sublinearity is sufficient to conclude that \eqref{def:covariance} is non-degenerate.

First, by Cram\'{e}r-Wold's device, it is sufficient to show that for every $v\in \mathbb{R}^d$ and $\mathbb{P}$-a.e. $\omega$, $v\cdot M^{(n), \omega}$ converges weakly to a one-dimensional Brownian motion with variance $v\cdot \Sigma^2v =  2\norm{v\cdot \Phi}^2_{\mathrm{cov}}.$ Similar to the case of $M^{i, \omega}$, $v\cdot M^{(n), \omega}$ is locally square integrable local martingale, and for $\varepsilon\geq 0$, the quadratic variation process 
   \begin{align*}
       \sigma_t^\varepsilon[v\cdot M^{(n), \omega}]:= \sum_{0\leq s\leq t}\abs{v\cdot \Delta M_s^{(n), \omega}}^2 \mathbbm{1}_{\{\abs{v\cdot \Delta M_s^{(n), \omega}}>\varepsilon\}}
   \end{align*}
   has a unique compensator,  given\footref{ft:Levy-system-theorem} by
   \begin{align}
      \label{eqn:compensator}  
      \tilde{\sigma}_t^\varepsilon[v\cdot M^{(n), \omega}]:= \int_0^{n^2t} \sum_{z\in \mathcal{N}} c^\omega(\tau_{X_s}, z)\abs{v\cdot \frac{1}{n}\Phi(\tau_{X_s}\omega, z)}^2 \mathbbm{1}_{\{ \abs{v \cdot \frac{1}{n}\Phi(\tau_{X_s}\omega, z)}>\varepsilon\}}\dd s
   \end{align}
  
By Helland's local martingale FCLT \cite[Theorem 5.1a]{Hel82}, the FCLT holds for $v\cdot M^\omega$ if following two conditions hold:
\begin{enumerate}[(H1)]
    \item $\tilde{\sigma}_t^0[v\cdot M^{(n), \omega}]\xrightarrow{n\to \infty}2t\norm{v\cdot \Phi}^2_{\mathrm{cov}}$ in $\Prob_0^\omega$, 
    \item $\tilde{\sigma}_t^\varepsilon[v\cdot M^{(n), \omega}]\xrightarrow{n\to \infty} 0$, in $\Prob_0^\omega$ and for all  $\varepsilon>0. $ 
\end{enumerate}
Clearly, (H1) follows immediately by applying to \eqref{eqn:compensator} the time-invaraince and -ergodicity (Proposition \ref{prop:time-ergodicity-and-non-explosion}). 
For (H2), using the same argument and the monotone convergence theorem, we have for $\varepsilon>0$, $n\geq K$, 
\begin{align*}
       \tilde{\sigma}_t^\varepsilon[v\cdot M^{(n), \omega}] & \leq   \int_0^{n^2t} \sum_{z\in \mathcal{N}} c^\omega(\tau_{X_s}, z)\abs{v\cdot \frac{1}{n}\Phi(\tau_{X_s}\omega, z)}^2 \mathbbm{1}_{\{ \abs{v \cdot \frac{1}{K}\Phi(\tau_{X_s}\omega, z)}>\varepsilon\}}\dd s \\
       & \xrightarrow{n\to \infty}2t\norm{(v\cdot \Phi)\mathbbm{1}_{\{\abs{v\cdot \Phi}>K\varepsilon}}^2_{\mathrm{cov}}\\
       & \xrightarrow{K\to \infty}0, \quad \mathbb{P}\times \Prob_0^\omega\text{-a.s.}
\end{align*}
Now we check the non-degeneracy, which shall follow from Assumption \ref{ass:P}(ii) and space-averaged $\ell^1$-sublinearity (a consequence of Proposition \ref{prop:space-averaged-sublinearity}).
Suppose for some $a\in \mathbb{R}^d$, with $\norm{a}_{\infty}=1$, $a\cdot \Sigma^2a$=0. Then 
\begin{align*}
    0 = \sum_{i,j=1}^da_ia_j\mathbb{E}\Big[\sum_{z\in \mathcal{N}}c_{\mathrm{s}}^{\omega}(0, z)\Phi^i(\omega, z)\Phi^j(\omega,z) \Big]
    = \mathbb{E}\Big[\sum_{z\in \mathcal{N}} c_{\mathrm{s}}^{\omega}(0,z)\big(\sum_{i=1}^da_i\Phi^i(\omega, z)\big)^2 \Big].
\end{align*}
Together with the cocycle property and \eqref{ass:symmetric-ellipticity-2}, which is equivelent to \eqref{ass:symmetric-ellipticity} in Assumption \ref{ass:P}-(ii), we have $a\cdot \Phi(\omega, x)=0$, for $\mathbb{P}$-a.e. $\omega$ and  all $x\in \mathbb{Z}^d$. Since $\Pi=\chi + \Phi$,  $a\cdot \Pi= a \cdot \chi$. Moreover, 
\begin{align}
     \frac{1}{\abs{B(n)}} \sum_{x\in B(n)}  \frac{1}{n} \abs{ a\cdot\Pi(\omega, x)} & =  \frac{1}{\abs{B(n)}} \sum_{x\in B(n)} \frac{1}{n}\abs{ a\cdot \chi(\omega, x)} \label{eq:degeneracy}\\
     & \leq \norm{a}_{\infty}\cdot \frac{1}{\abs{B(n)}} \sum_{x\in B(n)} \frac{1}{n} \abs{1 \cdot \chi(\omega, x)} \xrightarrow[]{n\to \infty}0, \nonumber
\end{align}
where the convergence is due to space-averaged $\ell^1$-sublinearity of the corrector.
Notice that 
\begin{align*}
    \text{LHS of }\eqref{eq:degeneracy} & =  \frac{1}{\abs{B(n)}} \sum_{x\in B(n)}  \abs{ a\cdot \frac{x}{n}} \\
    & =  \frac{1}{\abs{B_{\frac{1}{n}\mathbb{Z}^d}(1)}} \sum_{y\in B_{\frac{1}{n}\mathbb{Z}^d}(1)} \abs{ a\cdot y} \xrightarrow[]{n\to \infty} \fint_{y\in B_{\mathbb{R}^d}(1)} \abs{ a\cdot y} \dd y, 
\end{align*}
where $B_{\frac{1}{n}\mathbb{Z}^d}(1):=\{y\in \frac{1}{n}\mathbb{Z}^d: \norm{y}_\infty\leq 1\}$ and $B_{\mathbb{R}^d}(1):=\{y\in \mathbb{R}^d: \norm{y}_\infty \leq 1\}$.
As the integrand is a continuous function, we obtain $\abs{ a\cdot y} =0, \forall y\in  B_{\mathbb{R}^d}(1)$. However, this is only possible for $a=0$. Therefore we conclude that $\Sigma^2$ is non-degenerate. 
\end{proof}
\begin{remark} For alternative proof related to Lemma \ref{lem:sublinearCorrector_martingaleFCLT}-(ii), see \cite[Proposition 5.9]{AS23} for  martingale QFCLT, and see \cite[Proposition 2.5]{DNS18} for the non-degeneracy of $\Sigma^2$.
\end{remark}

\subsubsection{Vanishing corrector}
\begin{proof}[Proof of Lemma \ref{lem:sublinearCorrector_martingaleFCLT}-(iii)]
    It follows from Lemma \ref{lem:sublinearCorrector_martingaleFCLT}-(i), -(ii), and an application of Doob's maximal inequality. The proof is standard,  see for instance, \cite[Proposition 2.13]{ADS15}.
\end{proof}



\subsection{Proof of Theorem \ref{thm:QFCLT}}

Finally, we are ready to conclude the main result. 
\begin{proof}[Proof of Theorem \ref{thm:QFCLT}]
    The QFCLT for $X$ follows form Lemma \ref{lem:corrector_construction} and Lemma \ref{lem:sublinearCorrector_martingaleFCLT}.
\end{proof}
\section{Maximal inequality for harmonic functions}
\label{sec:max_ineq}
In this section we prove Proposition \ref{prop:maximum-inequality}. To align with the p-q moment condition in Assumption \ref{ass:p-q-moment}, we consider weighted graph inequalities of Proposition \ref{prop:weighted-graph}. These results follow from Proposition \ref{prop:graph}, which states well-known graph inequalities for the Euclidean lattice $(\mathbb{Z}^d, E_d)$. To derive the desired maximal inequality, we require energy estimate in the form of \eqref{eqn:energy-est}, that will be supplied to an iteration scheme, either Moser's (see \cite[Section 3]{ADS15}), or De Giorgi's. We choose the De Giorgi approach here, as it is robust in generalising the maximal inequality for  caloric functions, which is an important ingredient for the local limit theorem (see \cite{ADS16}).  Moreover, while the energy estimate is trivial for the symmetric case (i.e. $c^\omega = c_{\mathrm{s}}^\omega)$, it becomes non-trivial when the symmetry is absent, as the integration by parts principle is no longer applicable. This technical challenge is resolved in Proposition \ref{prop:energy-est} by the finite cycle decomposition, where the space-averaged $p$-th moment of $\mu^{(2), \omega}$ plays a role. Furthermore, we remark that, with minimal adaptation, the method in this section can be applied to derive the maximal inequalities on a general connected, locally finite, weighted graph that admits a finite cycle decomposition, provided the graph inequalities in Proposition \ref{prop:graph} are satisfied.

To proceed, we first state the following graph inequalities for the lattice graph $(\mathbb{Z}^d, E_d)$. 


%
%

%
\begin{prop}\label{prop:graph}
  For any $d \geq 2$, there exist constants $c_{\mathrm{reg}} , C_{\mathrm{reg}}, C_{\mathrm{S}_1}, C_{\mathrm{P}}\in (0, \infty)$ that depend at most on $d$,  
  such that for any $x \in \mathbb{Z}^d$, $n\in \mathbb{N}$ the following holds true:
  \begin{enumerate}[(i)]
  \item (Volume regularity) 
    \begin{align}\label{eq:reg:large_balls}
      c_{\mathrm{reg}}\, n^d \;\leq\; |B(x, n)| \;\leq\; C_{\mathrm{reg}}\, n^d.
    \end{align}

  \item (Sobolev inequality) 
  For every function $u\colon \mathbb{Z}^d \to\mathbb{R}$ with $\supp{u} \subset B(x, n)$,
    \begin{align} \label{eq:sobolev:ineq}
      \Norm{u}{d/(d-1), B(x, n)}
      \;\leq\;
      C_{\mathrm{S_1}}\, \frac{n}{\abs{B(x, n)}}\, \sum_{\{y, y'\} \in E_d} |u(y) - u(y')|.
    \end{align}

  \item (Weak Poincar\'{e} inequality) 
    For every function $u\colon \mathbb{Z}^d \to\mathbb{R}$
    \begin{align}\label{eq:poincare:ineq}
      \Norm{u - (u)_{B(x, n)}}{1, B(x, n)}
      \;\leq\;
      C_{\mathrm{P}}\, \frac{n}{\abs{B(x, n)})} 
      \sum_{\substack{y, y' \in B(x, n)\\ \{y, y'\} \in E_d}}
      \mspace{-24mu}|u(y) - u(y')|,
    \end{align}
    where $(u)_{B(x, n)} \ldef \frac{1}{|B(x, n)|} \sum_{y \in B(x, n)} u(y)$.
  \end{enumerate}
\end{prop}
\begin{proof}
    For the Euclidean lattice, $\abs{B(x,n)} = (2n+1)^d$, for $x\in \mathbb{Z}^d, n\in \mathbb{N}$. Proposition~\ref{prop:graph}-(i) is satisfied with $c_{\mathrm{reg}}=4^d, C_{\mathrm{reg}}=(\frac{5}{2})^d$.
    Provided Proposition~\ref{prop:graph}-(i), the Sobolev inequality in \eqref{eq:sobolev:ineq} follows from an isoperimetric inequality for large sets, see \cite[Proposition~3.5]{DNS18}.  The weak Poincar\'{e} inequality in \eqref{eq:poincare:ineq} follows from a (weak) relative isoperimetric inequality by applying a discrete version of the co-area formula, see \cite[Lemma~3.3.3]{Sa96}.
\end{proof}

We equip the graph $(\mathbb{Z}^d, E_d)$ with positive weight function $c: E_d \to (0, \infty)$, and define $\nu: \mathbb{Z}^d\to (0, \infty)$ by $\nu(y): = \sum_{ \{y, y'\}\in E_d} c(y,y')^{-1}$. Further for any $q\in [1, \infty)$, set \begin{align}
    \label{def:rho}
        \rho\equiv \rho(d, q):= \frac{d}{(d-2)+ d/q}.
    \end{align} 

It then follows from Proposition \ref{prop:graph}:

\begin{prop}[Weighted graph inequalities]
\label{prop:weighted-graph}For any $d\geq 2, q\in [1, \infty)$, there exist constants $C_{\mathrm{WS}}, C_{\mathrm{LP}}\in (0, \infty)$ that depend at most on $d, q$ such that for any $x\in \mathbb{Z}^d$ the following holds true: 
\begin{enumerate} [(i)]
    \item (Weighted Sobolev inequality on a large scale) 
    For every function $u: \mathbb{Z}^d \to \mathbb{R}$, and every cut-off function $\eta: \mathbb{Z}^d \to [0,1]$  with $\supp{\eta} \subset B(x, n)$,   
    \begin{align*}
        \norm{(\eta u)}^2_{2\rho, B(x, n)} \leq C_{\mathrm{WS}}\frac{n^2}{\abs{B(x,n)}}\norm{\nu}_{q, B(x,n)} \sum_{\{y,y'\}\in E_d}c(y,y')\abs{\eta u(y)- \eta u(y')}^2.
    \end{align*}
    \item (Local Poincar\'{e} inequality on a large scale) For every function $u: \mathbb{Z}^d \to \mathbb{R}$
    \begin{align*}
        & \norm{u-(u)_{B(x, n)}}^2_{2\rho, B(x, n)} \\
        & \qquad  \leq C_{\mathrm{LP}}\frac{n^{2}}{\abs{B(x,n)}}\norm{\nu}_{q, B(x,n)}  \sum_{\substack{ y,y'\in B(x, n) \\ \{y, y'\}\in E_d }}c(y,y')\abs{ u(y)- u(y')}^2.
    \end{align*}
\end{enumerate}
\end{prop}
\begin{proof} For (i), one uses Gagliardo-Nierenberg's argument, and apply the Cauchy-Schwarz and Hölder's inequality to separate $\nu$ from the last energy term, cf. \cite[Proposition 3.5]{ADS15}. For (ii), \cite[Theorem 4.1]{Cou96} implies that by Proposition \ref{prop:graph}-(i), the $\ell^1$-Poincaré inequality \eqref{eq:poincare:ineq} is equivalent to $(\frac{d}{d-1}, 1)$-Poincaré inequaity, i.e. \begin{align*}
    \norm{u-(u)_{B(x,n)}}_{\frac{d}{d-1}, B(x,n)} \leq C_{\mathrm{P}_1} \frac{n}{\abs{B(x, n)}} 
      \sum_{\substack{y, y' \in B(x, n)\\ \{y, y'\} \in E_d}}
      \mspace{-24mu}|u(y) - u(y')|, 
\end{align*}
for some constant $C_{P_1}$. The rest follows analogously as in  (i). 
\end{proof}

Now we consider weighted finite cycles on  the graph. Let $\Gamma_0=\{\vec{\gamma}_{i} : i \in \mathbb{N}\}$ be as described in Section \ref{subsec:model}, i.e. $\Gamma_0$ is a collection of finite, directed cycles on $\mathbb{Z}^d$ that satisfies \eqref{eqn:cycle_cover_nns}. Then, $\Gamma:= \{\vec{\gamma}+x, \vec{\gamma}\in \Gamma_0, x\in \mathbb{Z}^d\}$ satisfies %
\begin{align}\label{eq:def:cycle:condition}
  \bigl\{
    \{x, y\} :
    \exists\, \vec{\gamma} \in \Gamma
    \text{ s.th. either } (x, y) \in \vec{\gamma}
    \text{ or } (y, x) \in \vec{\gamma}
  \bigr\}
  \;=\;
  E_d,
\end{align}
Recall that, in each environment $\omega$,  to each cycle $\vec{\gamma}+x \in \Gamma$, we have associated cycle weight $\omega_{\vec{\gamma}}(x)\in [0, \infty)$. Since this section deals with deterministic weights, for simplicity, we denote the cycle weights by $\omega(\vec{\gamma})$ for every $\vec{\gamma} \in \Gamma$. 
The \emph{Dirichlet form}, $(\mathcal{E}_{\Gamma}^{\omega}, \mathcal{D}(\mathcal{E}_{\Gamma}^{\omega}))$ is then given by
\begin{align} \label{eq:def:DF}
  \left\{
    \begin{array}{rcl}
      \mathcal{E}_{\Gamma}^{\omega}(f, g) & \!\ldef\! &
      \sum_{\vec{\gamma} \in \Gamma} \omega(\vec{\gamma})\,
      \mathcal{E}_{\vec{\gamma}}(f, g)
      \\[1ex]
      \mathcal{D}(\mathcal{E}_{\Gamma}^{\omega})
      &\!\ldef\!&
      \bigl\{ f : \mathbb{Z}^d\to \mathbb{R} \text{ local} 
      \bigr\}
    \end{array}
 \right.,
\end{align}
where we define along each cycle $\vec{\gamma} \in \Gamma$ a quadratic form, $\mathcal{E}_{\vec{\gamma}}(f, g)$ by
\begin{align}
  \mathcal{E}_{\vec{\gamma}}(f, g)
  \;\ldef\;
  \sum_{(x, y) \in \vec{\gamma}} f(x)\, \bigl(g(x) - g(y)\bigr).
\end{align}
Notice that $\mathcal{E}_{\vec{\gamma}}(f-c, g) = \mathcal{E}_{\vec{\gamma}}(f, g)$ for any $f, g\colon \mathbb{Z}^d \to\mathbb{R}$, $\vec{\gamma} \in \Gamma$ and $c \in \mathbb{R}$. Indeed, since $\sum_{(x, y) \in \vec{\gamma}} \bigl(g(x) - g(y)\bigr) = 0$, it holds that
\begin{align*}
  \mathcal{E}_{\vec{\gamma}}(f, g)
  \;=\;
  \sum_{(x, y) \in \vec{\gamma}} \bigl(f(x) - c\bigr)\bigl(g(x) - g(y)\bigr)
  \;=\;
  \mathcal{E}_{\vec{\gamma}}(f-c, g).
\end{align*}

Recall that $\vec{E}_d$ denotes the set of all oriented edges of $E_d$. Let $c^\omega, c_{\mathrm{s}}^{\omega}, \mu^\omega, \nu^\omega, \mu^{(k), \omega}$ and generator $L^\omega$ be as defined in Section \ref{sec:model-results}.

In particular, we have $\mathcal{E}^\omega_{\Gamma}(f,g)= \inp{f}{-L^\omega g}$, for all local functions $f, g: \mathbb{Z}^d\to \mathbb{R}$. For convenience, we write $\mathcal{E}_{\vec{\gamma}}(f)\equiv \mathcal{E}_{\vec{\gamma}}(f, f), \mathcal{E}^\omega_{\Gamma}(f)\equiv \mathcal{E}^\omega_{\Gamma}(f, f).$

\vspace{1em}
Next, we proceed to show a maximal inequality in Proposition \ref{prop:maximum-inequality} for harmonic functions on the cyclic weighted graph $(\mathbb{Z}^d, \vec{E}_d, \Gamma, \omega)$, where, in particular, for every function $u$ that is locally harmonic, its maximum in a smaller space box is bounded by its space averaged norm of a lower order, with a prefactor that is controllable by assumption. The result is derived from an energy estimate in Proposition \ref{prop:energy-est} via the De Giorgi iteration method, cf. Proposition \ref{prop:iteration}. It is non-trivial to obtain the  energy estimate when $c^\omega$ is non-symmetric. The successful attainment relies on the representation of $c^\omega$ as a sum of non-negative cycle weights $\omega(\vec{\gamma})$. 

\vspace{1em}
Throughout, we consider for all $(x,y) \in \vec{E}_d, x \in V$, 
  \begin{align*}
    0< c_{\mathrm{s}}^{\omega}(x,y) <\infty, \quad \text{and} \quad  \mu^{(2),\omega}(x) <\infty.
  \end{align*}
This condition clearly holds under Assumption \ref{ass:p-q-moment} for $p=q=1$.

The main ingredient for the proof of Proposition \ref{prop:maximum-inequality}  is the following result. 
\begin{prop}[Energy estimate on a cyclic weighted Euclidean lattice]
\label{prop:energy-est}
Let $B$ be a finite subset of $\mathbb{Z}^d$. 
Consider a function $\eta: \mathbb{Z}^d\to \mathbb{R}$ with 
\begin{align*}
    \supp \eta \subset B, \quad 0\leq \eta \leq 1, \quad \eta \equiv 0  \text{ on } \partial B.
\end{align*}
Suppose $u: \mathbb{R} \times \mathbb{Z}^d \to \mathbb{R}$ is either non-negative $L^\omega$-subharmonic $(-L^\omega u\leq 0)$,  or $L^\omega$-harmonic $(L^\omega u = 0)$ on $B$. Then we have for $p \in [1,\infty]$, there exists a constant $C_{\mathrm{En}} $ such that for all $t\in I $, 
\begin{align}
\label{eqn:energy-est}
  \frac{\mathcal{E}_\Gamma^\omega(\eta u)}{\abs{B}}\leq C_{\mathrm{En}} \norm{\mu^{(2), \omega}}_{p, B} \, \norm{\nabla \eta}^2_{\ell_\infty(E)}  \norm{u^2}_{p_*,B}, 
\end{align}
with $p_*= p/(p-1)$ being the H\"{o}lder conjugate of $p$. 

     \end{prop}
     \begin{remark}
     \label{rm:overflow}
     It is easy to check that if $u$ is assumed in Proposition \ref{prop:energy-est}, then for any constant  $l\geq 0$, $(\abs{u}-l)_+$ is $L^\omega$-subharmonic on $ B$.  
     \end{remark} 
         
     \begin{proof}[Proof of Proposition \ref{prop:energy-est} ]

     First, to each $\vec{\gamma} \in \Gamma$, we associate an unweighted operator $I_{\vec{\gamma}}$ defined by $I_{\vec{\gamma}} g(x) := \sum_{(x, y)\in \vec{\gamma}}\nabla g(x,y)$, for all $g:\mathbb{Z}^d\to \mathbb{R}$, and with $\nabla g(x,y):=g(y)-g(x)$.
     Let $u, \eta$ be as assumed, and denote by $\dot{\gamma}$ the set of vertices in the cycle $\vec{\gamma}$. Then, 
     \begin{align}
      \label{eqn:energy_est_start}   
        0 \leq \inp{\eta^2 u}{L^\omega u} = \sum_{\vec{\gamma}\in \Gamma}\omega(\vec{\gamma})\inp{\eta^2 u}{I_{\vec{\gamma}}u}_{\dot{\gamma}}.
          \end{align}
     Our next step is to bound $\inp{\eta^2u}{I_{\vec{\gamma}}u}_{\dot{\gamma} }$ from above by the form of 
\begin{align*}
    \text{const}_1\cdot \mathcal{E}_{\vec{\gamma}}(\eta u) - \text{const}_2 \cdot F(\norm{\nabla \eta}^2_{\ell_\infty(\vec{E}_d)}\cdot u^2\restrict{\dot{\gamma} \cap B}).
\end{align*}
for some deterministic function $F$.

Let $\dot{\gamma}_B = \dot{\gamma}\cap B$, and $\vec{\gamma}_B$ denote the subset of oriented edges of $\vec{\gamma}$ with both endpoints in $B$.

Observe that 
\begin{align}
   &  \langle \, \eta^2u, I_{\dot{\gamma}} u\,\rangle_{\vec{\gamma}} \nonumber \\
    & \; = \sum_{(x,y)\in \vec{\gamma}} (\eta^2u) (x)  \nabla u (x,y) \nonumber \\
    & \;  = \sum_{(x,y)\in \vec{\gamma}}  \big(-(\eta^2u^2) (x)  + (\eta u)(x) (\eta u)(y) - (\eta u)(x) (\eta u)(y) + (\eta^2u)(x)u(y) \big) \nonumber \\
    & \;  = \underbrace{\inp{\eta u}{I_{\vec{\gamma}}(\eta u)}_{\dot{\gamma}_B}}_{:=A_1}  + \sum_{(x,y)\in \vec{\gamma}_B} (\eta u)(x)u(y)\nabla\eta(y,x)\nonumber  \\
    & \; = A_1 +  \frac{1}{\abs{\dot{\gamma}_B}} \sum_{\substack{(x,y)\in \vec{\gamma}_B \nonumber \\ z\in \dot{\gamma}_B}} \big( \underbrace{\nabla (\eta u)(z,x)u(y)}_{:=A_{21}(x,y,z)}   + (\eta u)(z) u(y) \big) \nabla \eta(y,x)\nonumber  \\
    & \; \stackrel{(*)}{=}  A_1 +  \frac{1}{\abs{\dot{\gamma}_B}}\sum_{\substack{(x,y)\in \vec{\gamma}_B \\ z\in \dot{\gamma}_B}} \big(A_{21}(x,y,z)  + \big(\eta(z)u(y)- (\eta u)(z)\big) u(z)\big) \nabla \eta(y,x) \nonumber \\
    & \;=  A_1 + \frac{1}{\abs{\dot{\gamma}_B}}\sum_{\substack{(x,y)\in \vec{\gamma}_B \\ z\in\dot{\gamma}_B}} \big(A_{21}(x,y,z)   + \underbrace{\nabla (\eta u)(z,y)  u(z)}_{:=A_{221}(y,z)}+ \underbrace{\nabla \eta (y, z) u(y) u(z)}_{:=A_{222}(y,z)} \big) \nabla \eta(y,x), \label{eqn:master_energy_est}
\end{align}
where $(*)$ is due to  $\sum_{(x,y)\in \vec{\gamma}_B} \nabla \eta (y,x) =0$. 
Notice that $$A_1= - \frac{1}{2}\inp{\nabla(\eta u)}{\nabla(\eta u)}_{\vec{\gamma}}= -\mathcal{E}_{\vec{\gamma}}(\eta u).$$ Moreover by Young's inequality, 
\begin{align*}
  \abs{A_{21}(x,y,z)} &\leq  \frac{1}{2}\sum_{(x',y')\in \vec{\gamma}_B}\abs{ \nabla (\eta u)(x', y')} \abs{u(y)}  \leq  \frac{\varepsilon}{2}\mathcal{E}_{\vec{\gamma}} (\eta u) 
    + \frac{1}{4\varepsilon}\abs{\vec{\gamma}_B} u^2(y), \\
    \abs{A_{221}(y,z)} & \leq \frac{1}{2}\sum_{(x', y')\in \vec{\gamma}_B }\abs{\nabla (\eta u)(x', y')}  \abs{u(z)}  \leq \frac{\varepsilon}{2}\mathcal{E}_{\vec{\gamma}}(\eta u)  + \frac{1}{4\varepsilon }\abs{\vec{\gamma}_B}  u^2(z).
    \end{align*}
Furthermore, 
\begin{align*}
 \abs{A_{222}(y,z)} & \leq \frac{1}{2} \abs{\vec{\gamma}_B} \, \norm{\nabla \eta}_{\ell_\infty(\vec{E}_d)} \abs{u(y)}\abs{u(z)}.
\end{align*}
Substituting  above estimates into \eqref{eqn:master_energy_est} with the choice $\varepsilon = (2\abs{\vec{\gamma}_B}\norm{\nabla \eta}_{\ell_\infty(\vec{E}_d)})^{-1}$, we have 
\begin{align}
\label{eqn:energy_est_unweighted}
    \inp{\eta^2u}{I_{\vec{\gamma}} u}_{\vec{\gamma}} \leq -\frac{1}{2}\mathcal{E}_{\vec{\gamma}}(\eta u) + (\abs{\vec{\gamma}_B}^2+\frac{1}{2}\abs{\vec{\gamma}_B})\norm{\nabla \eta}_{\ell_{\infty(\vec{E}_d)}}^2 \sum_{x\in \dot{\gamma}_B}u^2(x).
\end{align}
Combining with \eqref{eqn:energy_est_start}, and by Hölder's inequality, we obtain
\begin{align*}
    \mathcal{E}^\omega_{\Gamma}(\eta u ) & \leq  \norm{\nabla \eta}_{\ell_{\infty(\vec{E}_d)}}^2 \sum_{x\in B} u^2(x) \sum_{\vec{\gamma} \in \Gamma} \big(2\abs{\vec{\gamma}}^2 +\abs{\vec{\gamma}}\big)\omega(\vec{\gamma})\mathbbm{1}_{x\in \dot{\gamma}} \\
    & \leq  \norm{\nabla \eta}_{\ell_{\infty(\vec{E}_d)}}^2 \abs{B} \sum_{x\in B} u^2(x) \big(2\mu^{(2), \omega}(x)+ \mu^{(1), \omega}(x)\big) \\
    & \leq \frac{5}{2} \norm{\nabla \eta}_{\ell_{\infty(\vec{E}_d)}}^2 \abs{B} \norm{\mu^{(2), \omega}}_{p, B}\norm{u^2}_{p_*, B}.
\end{align*}
    By moving $\abs{B}$ to the LHS, we conclude \eqref{eqn:energy-est}.
     \end{proof}

     \begin{lemma}[De Giorgi's iteration method]
     \label{prop:iteration}
         Let $f:[0, \infty)\times [0, \infty) \to [0, \infty]$ be a non-negative function. Suppose it holds, that for all $l>k \geq 0$ and $ \sigma>\sigma' \geq 0$
         \begin{align*}
             f(l, \sigma')\leq \frac{C_{\mathrm{DG}}}{(\sigma - \sigma')^\alpha(l-k)^\beta} f(k, \sigma)^\gamma, 
         \end{align*}
         for some constants $C_{\mathrm{DG}}, \alpha, \beta >0$ and $\gamma >1$. Then, there exists a constant $K\equiv K(C, f, \sigma', \alpha, \beta, \gamma)$ such that 
         \begin{align*}
             f(l, \sigma')=0, \quad \forall l\geq K. 
         \end{align*}
         In particular, we may choose
         \begin{align}
        \label{eqn:max}
             K = f(0, \sigma)^{\frac{\gamma-1}{\beta}}C_{\mathrm{DG}}^{\frac{1}{\beta}}2^{\frac{(\alpha+\beta)}{\gamma-1}+1}(\sigma-\sigma')^{-\frac{\alpha}{\beta}}, 
         \end{align}
         
     \end{lemma}
     \begin{proof}
          cf. the iteration introduced for the proof of \cite[Theorem 4.1]{HL11}.
     \end{proof}

     \begin{proof}[Proof of Proposition \ref{prop:maximum-inequality}]
     Our strategy is to show firstly 
     \begin{align}
\label{eqn:DeGiorgi-criteria}
    \norm{(\abs{u}-l)_+^2}_{p_*, B(x, \sigma'n)} \leq C_1 \frac{ 
    \norm{\mu^{(2), \omega}}_{p, B(x, \sigma n)}\norm{\nu^\omega}_{q, B(x, \sigma n)}}{(\sigma - \sigma')^2 (l-k)^{2\gamma -2}}\norm{(\abs{u}-k)_+^2}^\gamma_{p_*, B(x, \sigma n)}, 
\end{align}
for some constant $C_1\equiv C_1(c_\mathrm{reg}, C_\mathrm{reg}, 
d, p, C_\mathrm{WS}, C_{\mathrm{En}})$. Then we apply Proposition \ref{prop:iteration}, and use the choice \eqref{eqn:max} to conclude \eqref{eqn:max-ineq}.

We start by defining a test function $\eta$ 
satisfying the following assumptions: 
\begin{align*}
    & \supp \eta \subset B(x, \sigma n), \quad \eta \equiv 1 \text{ on } B(x, \sigma'n), \quad \eta \equiv 0 \text{ on } \partial B(x, \sigma n).\\
\end{align*}
By interpolating $\eta$ 
linearly on $B(x, \sigma n)\setminus B(\sigma 'n)$,  
it follows
\begin{align}
\label{ineq:test-function-bounds}
    \norm{\nabla \eta}_{\ell^\infty(E)} \leq \frac{1}{(\sigma-\sigma')n}. 
\end{align}
Furthermore, we set $\delta = \rho/p_*$. By $1/p+1/q<2/d$ and \eqref{def:rho}, $\delta >1$. 

Following from Proposition \ref{prop:graph}-(i), $\frac{1}{2}\leq \sigma' \leq \sigma  \leq 1$, and Hölder's inequality, 
\begin{align*}
     \norm{(\abs{u}-l)_+^2}_{p_*, B(x, \sigma'n)} & \leq C_2  \norm{(\eta (\abs{u}-l)_+)^2}_{p_*, B(x, \sigma n)} \\
     & \leq C_2\norm{(\eta (\abs{u}-l)_+)^2)}_{\rho, B(x, \sigma n)} \norm{\mathbbm{1}_{\{\abs{u}>l\}}}^{\frac{1}{\delta*}}_{p_*, B(x, \sigma n)}
\end{align*}
holds for large $n$, with $C_2\equiv C_2(c_\mathrm{reg}, C_\mathrm{reg}, p_*, d)$ 
and $\delta_*$ being the Hölder conjugate of $\delta$.
Moreover, combining the  weighted Sobolev inequality(Proposition \ref{prop:weighted-graph}-(i)) and the energy estimate \eqref{eqn:energy-est} (cf. Remark \ref{rm:overflow}), we derive 
\begin{align*}
     &  \norm{(\eta (\abs{u}-l)_+)^2}_{\rho, B(x, \sigma n)} \\
    &\qquad \leq 2C_\mathrm{WS}\, n^2\norm{\nu^\omega}_{q, B(x, \sigma n)}\frac{\mathcal{E}^\omega (\eta (\abs{u}-l)_+)}{\abs{B(x, \sigma n)}} \\
    & \qquad\leq 2C_\mathrm{WS} C_{\mathrm{En}}n^2\norm{\nu^\omega}_{q, B(x, \sigma n)} \norm{\mu^{(2), \omega}}_{p, B(x, \sigma n)} \norm{\nabla \eta}^2_{\ell^\infty(E)}   \norm{(\abs{u}-l)_+^2}_{p_*,  B(x, \sigma n)}  \\
        & \qquad\stackrel{\eqref{ineq:test-function-bounds}}{\leq} C_3 \frac{\norm{\mu^{(2), \omega}}_{p, B(x, \sigma n)}\norm{\nu^\omega}_{q, B(x, \sigma n)}}{(\sigma-\sigma')^2} \norm{(\abs{u}-l)_+^2}_{p_*, B(x, \sigma n)},
\end{align*}
where $C_3=2C_{\mathrm{WS}}C_{\mathrm{En}}.$
In addition, Markov's inequality implies, for $0\leq k\leq l$,
\begin{align*}
    \norm{\mathbbm{1}_{\{\abs{u}>l\}}}_{p_*, B(x, \sigma n)} = \norm{\mathbbm{1}_{\{(\abs{u}-k)_+>l-k\}}}_{p_*, B(x, \sigma n)}\leq \frac{\norm{(\abs{u}-k)^2_+}_{p_*, B(x, \sigma n)}}{(l-k)^2}.
\end{align*}
Together with previous estimates, we derive \eqref{eqn:DeGiorgi-criteria} with $\gamma = 1+\frac{1}{\delta_*}, $ and $C_1=C_2C_3.$

Next, we observe that the criteria of Proposition \ref{prop:iteration} is satisfied, with 
\begin{align*}
    & f(m, \sigma)= \norm{(\abs{u}-m)^2_+}_{p_*, B(x, \sigma n)}, \quad \forall m, \sigma \geq 0; \\
    &  \alpha =2, \quad \beta = \frac{2}{\delta_*}, \quad \gamma = 1+\frac{1}{\delta_*};\\
    & C_{\mathrm{DG}}= C_1 \norm{\mu^{(2), \omega}}_{p, B(x, \sigma n)}\norm{\nu^\omega}_{q, B(x, \sigma n)}.
\end{align*}
By evaluating \eqref{eqn:max}, the maximal inequality \eqref{eqn:max-ineq} follows by choosing
\begin{align*}
    \kappa = \frac{\delta^*}{2}, \qquad 
    C_{\mathrm{Max}}= 2^{4\kappa+3}C_1^\kappa. 
\end{align*}
     \end{proof}

\bibliographystyle{abbrv}
\bibliography{literature}

\end{document}